\documentclass[11pt]{amsart}

\usepackage[utf8]{inputenc}
\usepackage[T1]{fontenc}
\usepackage[english]{babel}
\usepackage{amsmath}
\usepackage{amsthm}
\usepackage{amssymb}
\usepackage{mathrsfs}
\usepackage{dsfont}
\usepackage{graphicx} % Required for inserting images
\usepackage[left=3cm,right=3cm,top=3cm,bottom=3cm]{geometry}
\usepackage[dvipsnames]{xcolor}
\usepackage[colorlinks=true, citecolor=MidnightBlue, linkcolor=MidnightBlue]{hyperref}
\usepackage{tikz-cd}
\usetikzlibrary{shapes}
\usetikzlibrary[patterns]
\usetikzlibrary{decorations.pathreplacing}
\usepackage{caption}
\usepackage{subcaption}
\usepackage{enumitem}
\setitemize{label=$\triangleright$}
\usepackage{eucal} %lettres calligraphiées
\usepackage{mlmodern}

\newcommand{\tomega}{\widetilde{\omega}}

\newcommand{\PP}{\mathbb{P}}
\newcommand{\NN}{\mathbb{N}}
\newcommand{\ZZ}{\mathbb{Z}}
\newcommand{\RR}{\mathbb{R}}
\newcommand{\CC}{\mathbb{C}}

\newcommand{\FF}{\mathbb{F}}
\newcommand{\CP}{\mathbb{CP}}

\newcommand{\tDDD}{\widetilde{\mathscr{D}}}

\newcommand{\WWW}{\mathscr{W}}

\newcommand{\bfl}{\mathbf{l}}

\newcommand{\Pfk}{\mathfrak{P}}

\newcommand{\mfk}{\mathfrak{m}}
\newcommand{\pfk}{\mathfrak{p}}

\newcommand{\tfk}{\mathfrak{t}}
\newcommand{\bfk}{\mathfrak{b}}

\newcommand{\sfk}{\mathfrak{s}}

\newcommand{\sfb}{\mathsf{b}}
\newcommand{\sfe}{\mathsf{e}}
\newcommand{\sff}{\mathsf{f}}
\newcommand{\sft}{\mathsf{t}}
\newcommand{\sfs}{\mathsf{s}}
\newcommand{\sfw}{\mathsf{w}}

\newcommand{\sfB}{\mathsf{B}}
\newcommand{\sfT}{\mathsf{T}}
\newcommand{\sfW}{\mathsf{W}}
\newcommand{\sfS}{\mathsf{S}}

\renewcommand{\div}{\mathrm{div}}
\newcommand{\codeg}{\mathrm{codeg}}
\newcommand{\Area}{\mathrm{Area}}

\newcommand{\D}{\mathcal{D}}
\newcommand{\tD}{\widetilde{\mathcal{D}}}
\newcommand{\E}{\mathcal{E}}
\newcommand{\tE}{\widetilde{\mathcal{E}}}
\renewcommand{\P}{\mathcal{P}}
\renewcommand{\L}{\mathcal{L}}

\renewcommand{\S}{\mathcal{S}}

\renewcommand{\O}{\mathcal{O}}

\newcommand{\dd}{\mathrm{d}}

\newcommand{\dsum}{\displaystyle\sum}
\newcommand{\dprod}{\displaystyle\prod}
\newcommand{\gmax}{g_{\max}}
\renewcommand{\leq}{\leqslant}
\renewcommand{\geq}{\geqslant}
\renewcommand{\tilde}{\widetilde}
\renewcommand{\bar}{\overline}
\newcommand{\ie}{i.e. }
\newcommand{\eg}{e.g. }

\def\floor (#1) at (#2,#3); {
    \node[draw,ellipse, minimum width=1cm, minimum height = 0.6 cm] (#1) at (#2,#3) {$\bullet$} ;
}

\def\ufloor (#1) at (#2,#3) (#4); {
    \node[draw,ellipse, minimum width=1cm, minimum height = 0.6 cm] (#1) at (#2,#3) {\scriptsize #4} ;
}

\def\marked (#1) to (#2) pos=#3 in=#4 out=#5; {
   \draw (#1) to[out=#5,in=#4] node[pos=#3] {$\bullet$} (#2) ;
}

\def\dashedmarked (#1) to (#2) pos=#3 in=#4 out=#5; {
   \draw[dashed] (#1) to[out=#5,in=#4] node[pos=#3] {$\bullet$} (#2) ;
}

\def\leftmarked (#1) to (#2) pos=#3 in=#4 out=#5 w=#6; {
   \draw[line width=2pt] (#1) to[out=#5,in=#4] node[pos=#3] {$\bullet$} node[midway,left] {$#6$} (#2) ;
}

\def\wlmarked (#1) to (#2) pos=#3 in=#4 out=#5 w=#6; {
   \draw (#1) to[out=#5,in=#4] node[pos=#3] {$\bullet$} node[midway,left] {$#6$} (#2) ;
}

\def\rightmarked (#1) to (#2) pos=#3 in=#4 out=#5 w=#6; {
   \draw[line width=2pt] (#1) to[out=#5,in=#4] node[pos=#3] {$\bullet$} node[midway,right] {$#6$} (#2) ;
}

\def\doublemarked (#1) to (#2) pos=#3 in=#4 out=#5; {
   \draw[line width=2pt] (#1) to[out=#5,in=#4] node[pos=#3] {$\bullet$} (#2) ;
}

\newcommand{\modulo}{\ \mathrm{mod}\ }
\newcommand{\ang}[1]{\langle #1\rangle}

\renewcommand{\bot}{\mathrm{bot}}
\renewcommand{\top}{\mathrm{top}}

\newcommand{\Eff}{\mathrm{DAmp}}

\newcommand{\bino}[2]{\begin{pmatrix} #1 \\ #2 \\ \end{pmatrix}}
\newcommand{\sbino}[2]{\left(\begin{smallmatrix} #1 \\ #2 \\ \end{smallmatrix}\right)}

\newtheorem{theo}{Theorem}[section]
\newtheorem*{theom}{Theorem}
\newtheorem{prop}[theo]{Proposition}

\newtheorem{lem}[theo]{Lemma}

\newtheorem{conj}{Conjecture}[section]

\theoremstyle{definition}
\newtheorem{defi}[theo]{Definition}

\theoremstyle{remark}
\newtheorem{remark}[theo]{Remark}
\newenvironment{rem}[1]{
    \begin{remark}#1}{
    \xqed{\blacklozenge}\end{remark}
}

\theoremstyle{remark}
\newtheorem{example}[theo]{Example}
\newenvironment{expl}[1]{
    \begin{example}#1}{
    \xqed{\lozenge}\end{example}
}

\newcommand{\xqed}[1]{
    \leavevmode\unskip\penalty9999 \hbox{}\nobreak\hfill
    \quad\hbox{\ensuremath{#1}}}

\newcommand{\nocontentsline}[3]{}
\newcommand{\tocless}[2]{\bgroup\let\addcontentsline=\nocontentsline#1{#2}\egroup}

\makeatletter
\def\l@subsection{\@tocline{2}{0pt}{2.5pc}{5pc}{}}
\makeatother

\makeatletter
\renewcommand{\l@section}{\@tocline{1}{0pt}{10pt}{1pc}{\bfseries}}
\makeatother

\title[Asymptotic computations of tropical refined invariants]{Asymptotic computations of tropical refined invariants in genus 0 and 1} 
\author{Thomas Blomme, Gurvan Mével}
\address{Université de Neuchâtel, rue \'Emile Argan 11, Neuchâtel 2000, Suisse}
\email{thomas.blomme@unine.ch}

\address{CNRS \& Nantes Université, UMR 6629 Laboratoire de Mathématiques Jean Leray, 2 rue de la Houssinière, F-44322 Nantes Cedex 3, France}
\email{gurvan.mevel@univ-nantes.fr}


\begin{document}

\begin{abstract}
Block and Göttsche introduced a Laurent polynomial multiplicity to count tropical curves. Itenberg and Mikhalkin then showed that this multiplicity leads to invariant counts called tropical refined invariants. Recently, Brugallé and Jaramillo-Puentes studied the polynomiality properties of the coefficients of these invariants and showed that for fixed genus $g$, the coefficients ultimately coincide with polynomials in the homology class of the curves we look at. We call the generating series of these polynomials asymptotic refined invariant. In genus 0, the asymptotic refined invariant has been computed by the second author in the $h$-transverse case.

In this paper, we give a new proof of the formula for the asymptotic refined invariant for $g=0$ using variations on the floor diagram algorithm. This technique allows also to compute the asymptotic refined invariant for $g=1$. The result exhibits surprising regularity properties related to the generating series of partition numbers and quasi-modular forms.
\end{abstract}

\maketitle

\tableofcontents

\section{Introduction}

\subsection{Context}

\subsubsection{Enumerative invariants and polynomiality.}

Given some points in the complex plane, the problem of determining how many curves of fixed degree and genus pass through these points is a well-known question, which generalizes to other surfaces. Given a non-singular complex surface $X$ equipped with a sufficiently ample line bundle $\L$, curves on $X$ may be obtained as zero-sets of sections of $\L$. For a non-negative integer $g$, how many such curves of genus $g$ pass through $c_1(X)\cdot\L-1+g$ points on $X$ ? We denote by $N_g^X(\L)$ this number, which is also known as the degree of the corresponding Severi variety. As a consequence of the adjunction formula, we could consider a dual question: given $\delta$ a number of nodes, what is the number $N_X^\delta(\L)$ of $\delta$-nodal curves passing through the appropriate number of points on $X$ ?

Although determining these numbers is a difficult problem, some recursive formulas have been proved in the 90's by Kontsevich \cite{kontsevich1994gromov-witten} in the specific case of rational curves, and by Caporaso-Harris \cite{caporaso1998counting}. In the same decade, Göttsche conjectured in \cite{gottsche1998conjectural} the number $N_X^\delta(\L)$ to behave polynomially when $X,\delta$ are fixed and $\L$ varies. This conjecture has first been proved by Tzeng \cite{tzeng2012proof}, then also by Kool, Shende and Thomas \cite{kool2011short}, and is as follows: for any $\delta$, there exists a universal polynomial $P_\delta \in \CC[x,y,z,t]$ such that for any non-singular complex algebraic surface with a sufficiently ample line bundle $\L$, one has
\[N_X^\delta(\L) = P_\delta(\L^2,c_1(X)\cdot\L,c_1(X)^2,c_2(X)).\]
Göttsche conjecture also states that the generating series of the $(P_\delta)_\delta$ is multiplicative, in that there exist formal series $A_1,\dots,A_4$ such that the generating series is $A_1^x A_2^y A_3^z A_4^t$, with explicit descriptions of some of the $A_i$.

\medskip

This polynomial behavior is not satisfied when we fix the genus $g$. For instance, Di Francesco and Itzykson \cite{di_francesco_quantum_1995} proved that $\log(N^{\CP^2}_0(d)) \sim 3d\log(d)$. However, Brugallé and Jaramillo-Puentes showed in \cite{brugalle2020polynomiality} that we recover polynomiality when looking instead at the coefficients of the \textit{tropical refined invariant}.

\subsubsection{Tropical refined invariants and their asymptotic}

Tropical refined invariants emanate from Mikhalkin's correspondence theorem \cite{mikhalkin2005enumerative} which enables the computation of $N_g^X(\L)$ for toric surfaces, transforming the previous algebraic problem into a combinatorial enumerative problem dealing with objects called \textit{tropical curves}. The correspondence theorem assigns some integer multiplicities to tropical curves. Block and G\"ottsche \cite{block2016refined} proposed to refine this multiplicity, yielding instead a Laurent polynomial (in a formal variable $q$) count of tropical curves, which interpolates between Gromov-Witten invariants for $q=1$ and tropical Welshinger invariants for $q=-1$. Itenberg and Mikhalkin \cite{itenberg2013block} proved that the enumeration using this new refined multiplicity indeed provides an invariant, known as tropical refined invariant and denoted by $BG^X_g(\L)(q)$.

In \cite{brugalle2020polynomiality}, Brugall\'e and Jaramillo-Puentes showed, in the case of Hirzebruch surfaces and (weighted) projective spaces, that for any fixed $i$ the coefficient of codegree $i$ of the tropical refined invariant is polynomial, providing that the line bundle is sufficiently ample with respect to $i$. Results of \cite{brugalle2020polynomiality} have been extended in genus 0 to any \textit{$h$-transverse} toric surface, including singular ones, by the second author \cite{mevel2023universal}. In sight of the multiplicativity part of G\"ottsche conjecture, \cite{mevel2023universal} also provides explicit formula for the generating series of the polynomials that give the coefficients of fixed codegree. In this paper, we provide a new proof for the generating series of the polynomials in the genus $0$ case, as well as formulas for the genus $1$ case, that may cast a new mystery toward the nature of tropical refined invariants.

\subsubsection{Interpretations and applications of tropical refined invariants}

Due to their original combinatorial definition, the meaning of tropical refined invariants in classical geometry remained mysterious for quite some time. Up to now, two main interpretations have been proved.
    \begin{itemize}
        \item Mikhalkin showed \cite{mikhalkin2017quantum} that in some situations, genus $0$ tropical refined invariants correspond to refined counts of real oriented curves according to the value of a so-called \textit{quantum index}.
        \item Bousseau proved \cite{bousseau2019tropical} that through the change of variable $q=e^{i\hbar}$, tropical refined invariants actually compute the generating series of (log-)Gromov Witten invariants with a $\lambda$-class. 
    \end{itemize}
Since then, results from \cite{mikhalkin2017quantum} have been generalized to genus $1$ and $2$ by Itenberg and Shustin \cite{itenberg2023real}, leading to real refined invariants. Unfortunately, the correspondence theorem does not relate them to tropical refined invariants defined using the Block-G\"ottsche multiplicity. Moreover, the tropical refined invariants involved in \cite{mikhalkin2017quantum} are of a different kind from the ones considered in the present paper, since its enumerative problem involves \textit{boundary constraints}. Although results from \cite{brugalle2020polynomiality} and the present paper do not apply to the invariants from \cite{mikhalkin2017quantum}, computed in \cite{blomme2019caporaso}, it would be interesting to know if the asymptotic results are true in this boundary setting, which would yield asymptotic information on the real invariants.

\medskip

Using \cite{bousseau2019tropical}, the results from \cite{brugalle2020polynomiality, mevel2023universal} may be interpreted as a subtle asymptotic statement about the (log-)GW invariants with a $\lambda$-class. The subtlety is due to the change of variable going from tropical refined invariants to these GW-invariants: $q=e^{i\hbar}$, \ie $q^{m/2}-q^{-m/2}=2i\sin\left(\frac{m\hbar}{2}\right)$, so that our results and those from \cite{brugalle2020polynomiality, mevel2023universal} give an asymptotic on the Fourier coefficients of the generating series of the GW-invariants with a $\lambda$-class. Furthermore, given that $q=e^{i\hbar}$ is also the change of variable occurring when relating GW-invariants and some Donaldson-Thomas invariants, there is also a possibility that the asymptotic of tropical refined invariants is actually a shadow from a property of some DT-invariants.

\medskip

Finally, in \cite{gottsche2014refined}, the refinement of tropical invariants is conjectured to correspond to the refinement of the Euler characteristic by the Hirzebruch genus $\chi_{-y}$ for some relative Hilbert scheme. Some work in this direction has been accomplished in \cite{nicaise2018tropical}. If such a correspondence was true, the asymptotic results from \cite{brugalle2020polynomiality, mevel2023universal} and the present paper may also be interpreted as asymptotic statements on the $\chi_{-y}$-genus of the relative Hilbert schemes.

    \subsection{Overview of results}

In this paper, we pursue the study of tropical refined invariants that was started in \cite{brugalle2020polynomiality} and \cite{mevel2023universal}, expanding the range of tools and possible computations. We also state some conjectures.

For the surface $X$, let $AR_{g,i}^X(\beta)$ be the asymptotic polynomial from \cite{brugalle2020polynomiality} giving the codegree $i$ coefficient of the genus $g$ refined invariant obtained by counting curves in the class $\beta$. Our results consist in an explicit computation of the $AR_{g,i}^X(\beta)$ in few particular cases, by determining their generating series in $i$ or $g$.

It is known from \cite{itenberg2013block} that the leading coefficient of the genus $g$ tropical refined invariant is $AR^X_{g,0} = \binom{\gmax}{g}$, where $\gmax$ is the genus of a non-singular curve in the class $\beta$ ; by the adjunction formula one has  $g_{\max}=1+\frac{1}{2}(\beta^2+\beta\cdot K_X)$. In other words, one has 
\[ \sum_{g\geq0} AR^X_{g,0} u^g = (1+u)^{\gmax}.\]
In this paper we provide a formula for the second term of the refined invariants.

\begin{theom}[{\bf \ref{theo-ARcodeg1}}]
For $X$ a smooth toric surface associated to a $h$-transverse and horizontal polygon, one has:
\[\dsum_{g\geq 0} AR_{g,1}^X u^g =  (1+u)^{\gmax} \left[-\beta^2 \dfrac{u^3}{(1+u)^3} +2(K_X\cdot\beta)\dfrac{u^2}{(1+u)^3} + \chi\dfrac{1}{1+u} - K_X^2 \dfrac{u}{(1+u)^3} \right] .\]
\end{theom}

We then give a result concerning generating series when summing over $i$: let $AR_g^X(\beta)=\sum_{i=0}^\infty AR^X_{g,i}(\beta) x^i$ be this series, which we call the \textit{genus $g$ asymptotic refined invariant}.

In \cite{mevel2023universal}, the second author proved that $AR^X_0(\beta)= p(x)^\chi$, where $p(x)$ is the generating series of the partition numbers. However, the method used there does not seem to be manageable to deal with higher genus. With a slightly different point of view, we give in Theorem \ref{theo-AR-genus0-general} a different proof of this fact. The interest is that this point of view also allows to perform the computation in genus 1. A priori, the polynomiality behavior for general $h$-transverse toric surface has not been proven in \cite{brugalle2020polynomiality}, but in the genus 0 and 1 case it actually follows from our computations.

\begin{theom}[{\bf \ref{theo-AR-genus1-general}}]
For $X$ a non-singular toric surface associated to a $h$-transverse and horizontal polygon and with Euler characteristic $\chi$, the genus 1 asymptotic refined invariant satisfies

\[ AR^X_1(\beta) = p(x)^\chi\left( \gmax-12E_2(x) \right), \]
where $p(x)=\prod_{j=1}^\infty\frac{1}{1-x^j}$ is the generating series of the partition numbers and $E_2(x)=\sum_{a=1}^\infty \sigma_1(a)x^a$ is the first Eisenstein series.
\end{theom}

To prove the result we use floor diagrams, defined in \cite{brugalle2007enumeration,brugalle2008floor} and adapted in the refined setting in \cite{block2016fock}. We start with the case of Hirzebruch surfaces. It turns out most technicalities occur in the latter. It is then quite easy to obtain the result for other $h$-transverse toric surfaces by using the computation we did in the Hirzebruch case.

\subsection{Future directions and conjectures}

The form taken by the generating series for fixed genus suggests a nice but subtle regularity of the asymptotic of the refined invariants. However, we are for now limited by the computational techniques at our disposal, as the complexity of the computations increases quite fast with the genus or the codegree. We yet suggest a few possible generalizations that are either conjectures or potential results that the techniques hereby presented should be sufficient to prove for the brave reader but that we ultimately decided to leave out to avoid rendering this already quite technical paper even more technical.

First, although the results are stated for $h$-transverse polygons (especially the generating series for genus $0$ and $1$), we expect them to be true for all polygons yielding smooth toric surfaces. To get that, one way is to relate the refined invariants of a toric surface $X$ and its blow-up at a torus fixed point, using floor diagrams or equivalently a version of the tropical Caporaso-Harris formula \cite{gathmann2007caporaso} near the corner/side corresponding to the torus fixed point/exceptional divisor. Although the Caporaso-Harris formula \cite{gathmann2007caporaso} is wrong in a general toric surface, we expect the latter to be true for sufficiently ample divisor classes in finite codegree.

The method of computation of the present paper generalizes for $h$-transverse polygons having some singular corners (\ie the primitive direction vectors of the adjacent edges do not form an integral basis). Such corners lead to $A_n$-singularities in the corresponding toric surface. It is already proven in \cite{mevel2023universal} that for such toric varieties, the generating series for $g=0$ is actually $\prod_C p(x^{k_C})$, where the product runs over the corners $C$ of the polygon, and $k_C$ is the determinant of the adjacent primitive direction vectors at the corner. This generalizes the smooth case, where each of the $\chi$ corners has $k_C=1$. We also expect the expression $\prod_C p(x^{k_C})$ to hold for non-$h$-transverse toric surfaces with singular corners. We do not know the form of the series in the $g=1$ case.

With some time, the techniques presented should be able to provide a result for genus $2$ as well as codegree $2$, but at the cost of many lengthy computations. Given the above results, and that the result proven for $h$-transverse and non-singular toric surfaces should hold for any non-singular toric surface, we conjecture the following.

\begin{conj}\label{conj}
For non-singular surface $X$, the asymptotic refined invariant $AR^X_g$ has the following form:
$$AR^X_g(\beta)=p(x)^{\chi}\left( \bino{g_{\max}}{g}+Q^X_g(\beta^2,\beta\cdot K_X)\right),$$
where $Q^X_g$ is a polynomial of degree at most $g$ in $\beta^2$ and $\beta\cdot K_X$, whose coefficients are quasi-modular forms in the $x$ variable vanishing at $0$.
\end{conj}

As $g_{\max}$ is a polynomial function in $\beta^2$ and $\beta\cdot K_X$, so is $\sbino{g_{\max}}{g}$. The constant term of the refined invariant has already been computed in \cite{itenberg2013block} and is indeed $\sbino{g_{\max}}{g}$, so that the conjecture is true modulo $x$.

\medskip

The shape given in the conjecture emanates from computations in genus $0$ and $1$. To give more support to the quasi-modularity claim, we also have the following: in a future paper, we prove that the asymptotic refined invariants for Abelian surfaces is a polynomial in $\beta^2$ with coefficients in $\ZZ[\![x]\!]$ which are quasi-modular forms. This supports the mysterious appearance of quasi-modular forms in this setting. In the examples, the polynomial functions $AR^X_{g,i}(\beta)$ on the lattice $H_2(X,\ZZ)$ seem to be more precisely polynomials in $\beta^2$ and $\beta\cdot K_X$, justifying the form of the polynomial $Q^X_g$. We notice a more general conjecture would deal with the double generating series $\sum_{i,g}AR_{g,i}^Xu^gx^i$, and computations suggest to factor out $p(x)^\chi$ and $(1+u)^{\gmax}$. Theorems \ref{theo-ARcodeg1} and \ref{theo-AR-genus1-general} ensure that we have modulo $u^2$:
$$\sum_{i,g}AR_{g,i}^Xu^gx^i = p(x)^\chi(1+u)^{\gmax} \left[ 1-(\chi+K_X^2)uE_2(x)\right] \modulo u^2.$$

\medskip

Conjecture \ref{conj} treats the regularity of the asymptotic invariant for a fixed surface $X$. Similarly to the G\"ottsche conjecture, it would be interesting to study the dependence of $AR_g^X$ in the surface $X$. Hopefully, the polynomials $AR_{g,i}^X(\beta)$ are actually given by a universal polynomial $Q_{g,i}(\beta^2,\beta\cdot K_X,K_X^2,\chi(X))$. Theorems \ref{theo-ARcodeg1} and \ref{theo-AR-genus1-general} prove that it is the case for $g=0,1$ or $i=0,1$.

    \subsection{Organization of the paper}

    The precise setting of tropical refined invariants is recalled in Section \ref{sec-floor-diagrams-refined-invariants}. We also recall how to compute them with floor diagrams. Furthermore, we give a change of variables that transforms the symmetric Laurent polynomials into true polynomials in a new variable $x$, so that the codegree $i$ coefficient becomes the $x^i$ term. This allows for an easier formulation of the asymptotics. 

Section \ref{sec-series-fixed-codeg} is dedicated to the proof of Theorem \ref{theo-ARcodeg1}, which amounts to compute the generating series in the genus parameter.

Section \ref{sec-genus0} introduces \emph{words} that we will use in section \ref{sec-genus1}. We also prove Theorem \ref{theo-AR-genus0-general} using this tool to illustrate how it works.

Section \ref{sec-genus1} is devoted to the proof of Theorem \ref{theo-AR-genus1-general}. We start by explaining how to construct floor diagrams of genus $1$ from genus $0$ floor diagrams. This allows us to express the genus $1$ asymptotic refined invariant as a weighted sum over the genus $0$ floor diagrams. In both Sections \ref{sec-genus0} and \ref{sec-genus1} we start with the Hirzebruch case before going to the case of $h$-transverse non-singular toric surfaces.

We end with an appendix where we explain how to modify the calculations of this paper to deal with Göttsche-Schroeter invariants instead of Block-Göttsche invariants.

\subsection*{Acknowledgements} Part of the work was accomplished during the stays of T.B. in Nantes in March 2023 and of G.M. in Neuch\^atel in October 2023. We both would like to thank the other for the excellent working conditions.

T.B. is partially supported by the SNSF grant 204125. G.M. is supported by the CNRS and conducted his work within the France 2030 framework programme, Centre Henri Lebesgue ANR-11-LABX-0020-01.

\section{Floor diagrams and asymptotic refined invariants}
\label{sec-floor-diagrams-refined-invariants}

In this section, we recall the definition of tropical refined invariants and their computation using floor diagrams. We also reformulate the main result of \cite{brugalle2020polynomiality} to introduce \textit{asymptotic refined invariants}.

\subsection{Toric surfaces, homology classes and polygons.}

Let $N$ be a lattice and $M=\mathrm{Hom}(N,\ZZ)$ its dual. We denote by $M_\RR=M\otimes\RR$, $N_\RR=N\otimes\RR$ the associated real vector spaces. Following \cite{fulton1993toric}, a compact toric surface $X$ is obtained from a complete fan $\Sigma\subset N_\RR$, or from a polygon $\Delta \in M_\RR$ dual to $\Sigma$. The toric divisors of $X$ are in bijection with the rays of $\Sigma$. We assume that $X$ is smooth, \ie every cone of $\Delta$ is simplicial. The anticanonical class $-K_X$ is represented by the sum of toric divisors. The Euler characteristic $\chi$ of $X$ is equal to the number of rays of $\Sigma$.

Each complex curve in $X$ realizes some homology class $\beta\in H_2(X,\ZZ)$. Recall that $H_2(X,\ZZ)$ is endowed with the intersection product, which is non-degenerate by Poincar\'e duality. It is classical to show (see \cite{fulton1993toric}) that the homology group $H_2(X,\ZZ)$ is generated by the classes of the toric divisors in $X$. In particular, a class $\beta\in H_2(X,\ZZ)$ is fully determined by its intersection numbers with the toric divisors.

Consider the cone $\Eff(X)\subset H_2(X,\ZZ)$ of \emph{ample divisors classes} classes, i.e. classes $\beta$ such that $\beta\cdot D>0$ for any toric divisor $D$.

Let $\beta\in\Eff(X)$ be an ample divisor class. For any ray $\rho$ of the fan $\Sigma$, let $n_\rho \in N$ be a primitive vector such that $\rho = \RR_{\geqslant 0} \cdot n_\rho$, and let $D_\rho$ be the toric divisor corresponding to $\rho$. The multiset
\[ \mathrm{trop}(\beta)=\{n_\rho^{\beta\cdot D_\rho},\ \rho \text{ ray of } \Sigma \}, \]
where the notation $n_\rho^{\beta\cdot D_\rho}$ means that $n_\rho$ is taken $\beta\cdot D_\rho$ times, is called the \emph{tropical degree} of $\beta$. The sum of the vectors of the tropical degree is $0$ due to relations in $H_2(X,\beta)$ (see \cite{fulton1993toric}). Therefore, the tropical degree of $\beta$ determines a convex lattice polygon $\Delta_\beta\subset M_\RR$, with normal fan $\Sigma$, and such that the side dual to the ray $\rho$ has integer length $\beta \cdot D_\rho$.

\begin{rem}
If $\beta\in\Eff(X)$, the lattice polygon $\Delta_\beta$ gives an ample line bundle $\L_\beta$ on $X$ with Chern class $c_1(\L_\beta)\in H^2(X,\ZZ)$ Poincar\'e dual to $\beta$. A basis of sections of $\L_\beta$ is indexed by the lattice points of $\Delta_\beta$. For our purpose, we work with the class $\beta$ instead of the line bundle, as it would be the case in the setting of the G\"ottsche conjecture. These points of view are equivalent in the case of rational surfaces.
\end{rem}

\begin{expl}
    Take $\Sigma_{\PP^2}$ to be the complete fan in $\RR^2$ with three rays spanned by respectively $(0,-1),(-1,0)$ and $(1,1)$, giving as toric surface the projective plane $\PP^2$. Its second homology group $H_2(\PP^2,\ZZ)$ is isomorphic to $\ZZ$, spanned by the common class $L$ of any of the three toric divisors. The choice of $dL\in H_2(\PP^2,\ZZ) \simeq \ZZ $ yields the tropical degree
    $$\mathrm{trop}(dL)=\{ (0,-1)^d,(-1,0)^d,(1,1)^d\}.$$
    The associated polygon $\Delta_{dL}$ is $d$ times the unit triangle, which thus has vertices $(0,0),(d,0)$ and $(0,d)$ (see Figure \ref{fig-poly-CP2}), with associated line bundle $\O(d)$ on $\PP^2$.
\end{expl}

\begin{figure}[h!]
\begin{subfigure}[t]{0.49\textwidth}
\centering
	\begin{tikzpicture}[scale=0.7] 
		\draw (0,0) node {$\bullet$} ;
		\draw (0,0) node[below] {\scriptsize $(0,0)$} ;
		\draw (2,0) node {$\bullet$} ;
		\draw (2,0) node[below] {\scriptsize $(d,0)$} ;
		\draw (0,2) node {$\bullet$} ; 
		\draw (0,2) node[left] {\scriptsize $(0,d)$} ; 
		\draw (0,0) -- (2,0) -- (0,2) -- cycle ;
	\end{tikzpicture}
    \caption{$\Delta_{dL}$} 
    \label{fig-poly-CP2}
\end{subfigure}
\begin{subfigure}[t]{0.49\textwidth}
\centering
	\begin{tikzpicture}[scale=0.7] 
		\draw (0,0) node {$\bullet$} ;
		\draw (0,0) node[below] {\scriptsize $(0,0)$} ;
		\draw (1,0) node {$\bullet$} ;
		\draw (1,0) node[below] {\scriptsize $(b,0)$} ;
		\draw (5,2) node {$\bullet$} ; 
		\draw (5,2) node[right] {\scriptsize $(b+\delta a,a)$} ;
		\draw (0,2) node {$\bullet$} ; 
		\draw (0,2) node[left] {\scriptsize $(0,a)$} ; 
		\draw (0,0) -- (1,0) -- (5,2) -- (0,2) -- cycle ;
	\end{tikzpicture}
    \caption{$\Delta^\delta_{aE+bF}$} 
    \label{fig-poly-Hirzebruch}
\end{subfigure}
\caption{ }
\end{figure}

\begin{expl}
    Take $\Sigma_\delta$, with $\delta\geqslant 0$, to be the complete fan in $\RR^2$ with rays four spanned by respectively $(-1,0),(0,-1),(0,1)$ and $(1,-\delta)$. Let $F$, $E_\infty$ and $E=E_0$ be the classes of the divisors associated to the first three rays. They satisfy $E_0^2=\delta$, $F^2=0$ and $E_0\cdot F=1$. The toric surface associated to this fan is the Hirzebruch surface $\FF_\delta$, which has $H_2(\FF_\delta,\ZZ)\simeq\ZZ^2$. We have the relation 
    $$E_\infty=E_0-\delta F,$$
    so that $H_2(\FF_\delta,\ZZ)$ is spanned by $F$ and $E$. The toric divisor associated to the last ray lies also in the class $F$. Let $\beta=aE+bF\in\Eff(\FF_\delta)$ be an ample divisor class, which means that $a,b>0$. The associated tropical data is
    $$\mathrm{trop}(aE+bF)=\{ (0,1)^{b+\delta a},(0,-1)^{b},(-1,0)^a,(1,-\delta)^a \}.$$
    The corresponding polygon $\Delta^\delta_{aE+bF}$ is the trapezoid with vertices $(0,0)$, $(0,a)$, $(b,0)$ and $(b+\delta a,a)$, see Figure \ref{fig-poly-Hirzebruch}.
\end{expl}

\subsection{Floor diagrams and tropical refined invariants}

In \cite{itenberg2013block} it is shown that the count of genus $g$ tropical curves of degree $\mathrm{trop}(\beta)$ passing through a generic configuration of $-K_X\cdot\beta+g-1$ points with Block-G\"ottsche multiplicity does not depend on the choice of the points, yielding the tropical refined invariant $BG^X_g(\beta)(q)\in\ZZ[q^{\pm 1/2}]$. When $X$ comes from an $h$-transverse polygon, see the definition below, it is possible to compute them using the floor diagram algorithm from \cite{block2016fock}, which is the Block-G\"ottsche version of the floor diagram algorithm from \cite{brugalle2008floor}. This is the content of \cite[Theorem 4.3]{block2016refined}, stated below as Theorem \ref{theo-tropical-refined-invariant} that may be taken as a definition of $BG^X_{g,\beta}(q)$. We now recall how floor diagrams work.

\begin{defi}
    A convex lattice polygon $\Delta$ is said to be
    \begin{itemize}
        \item \emph{$h$-transverse} if any edge of $\Delta$ has a direction vector of the form $(\pm 1,0)$ or $(n,\pm 1)$ for some $n\in\ZZ$,
        \item \emph{horizontal} if it has a top and bottom horizontal edge,
        \item \emph{non-singular} if the associated toric surface is non-singular.
    \end{itemize}
Given $\Delta$ a lattice polygon we set the following notations.
\begin{itemize}
    \item The number of interior lattice points of $\Delta$ is $\gmax(\Delta) = |\mathring\Delta \cap\ZZ^2|$.
    \item The height of $\Delta$ is $a(\Delta)$. 
    \item The length of its top (resp. bottom) edge is $b^\top(\Delta)$ (resp. $b^\bot(\Delta)$); these may be 0 if $\Delta$ is not horizontal.
    \item $b_{\text{left}}(\Delta)$ (resp. $b_{\text{right}}(\Delta)$) is the multiset of integers $k$ appearing a number of times equal to the integral length of the side of $\Delta$ having $(-1,k)$ (resp. $(1,k)$) as outgoing normal vector.
\end{itemize}
When no ambiguity is possible we will simply use $\gmax$, $a$, $b^\top$, etc.
\end{defi}

For a toric surface associated to a $h$-transverse, horizontal and non-singular polygon, we have the following result.

\begin{lem}\label{lem-toric-surface-euler-char-self-intersection}
        Let $X$ be a smooth toric surface coming from an $h$-transverse, with a top and bottom horizontal sides corresponding to divisors $D_\top$ and $D_\bot$. We have
        $$D_\top^2+D_\bot^2+\chi(X)=4$$
        where $\chi(X)$ is the Euler characteristic of $X$. 
    \end{lem}

    \begin{proof}
        Using the Hodge numbers, the Euler characteristic of the sheaf $\O$ of holomorphic functions satisfies $\chi(\O)=1$. By Noether's formula, we know that $K_X^2+\chi(X)=12$. Moreover, as we have a top and bottom side, the toric surface is endowed with a map to $\PP^1$ provided by the first coordinate in the lattice of monomial. Let $F$ be the class of a fiber of this projection. The fiber over $0$ (resp. $\infty$) is the union of toric divisors coming from the left (resp. right) sides of the polygon. As divisor, it is a linear combination of the corresponding toric divisors. In the general case, the coefficients are the horizontal coordinates of the lattice vectors directing the corresponding rays in the fan. Since we are in the $h$-transverse case, the coefficients are $1$ so that
        \[ F=\sum_{D \text{ left side of }\Delta}D=\sum_{D \text{ right side of }\Delta}D.\]
        In particular, as the sum of all toric divisors is an anticanonical divisor, $D_\bot+D_\top+2F$ is an anticanonical divisor of $X$. As $D_\top\cdot D_\bot=0$, $F^2=0$, and $F\cdot D_{\top/\bot}=1$ we get
        \[ 12-\chi(X) = K_X^2 =  (D_\bot+D_\top+2F)^2  =  D_\bot^2+D_\top^2 + 4+4.  \]
    \end{proof}

An \emph{oriented graph} $\Gamma$ is a collection of vertices $V(\Gamma)$, bounded edges $E^0(\Gamma)$, sinks $E^\top(\Gamma)$ and sources $E^\bot(\Gamma)$. A bounded edge is a bivalent edge, \ie adjacent to two vertices. A sink (resp. source) is a univalent edge oriented outward (resp. inward) a vertex. The set of all edges is denoted by $E(\Gamma)$. A \emph{weight} on $\Gamma$ is an application $w : E(\Gamma) \to \ZZ_{>0}$. Every vertex $v$ has a \emph{divergence} which is the difference of the total weights entering and leaving the vertex, \ie  
\[ \div(v) = \dsum_{ \overset{e}{\to} v} w(e) - \dsum_{v \overset{e}{\to}} w(e).\]

\begin{defi}
    Let $\Delta$ be a $h$-transverse polygon. A \emph{floor diagram} $\D$ with Newton polygon $\Delta$ is the data of $(\Gamma,w,L,R)$, with $(\Gamma,w)$ a weigthed, connected, oriented and acyclic graph satisfying the following conditions :
    \begin{itemize}
        \item the graph $\Gamma$ has $a(\Delta)$ vertices called floors, $b^\top(\Delta)$ sinks and $b^\bot(\Delta)$ sources,
        \item all sinks and sources have weight $1$,
        \item the functions $L : V(\Gamma) \to b_{\text{left}}(\Delta)$ and $R : V(\Gamma) \to b_{\text{right}}(\Delta)$ are bijections such that for any vertex $v$ one has $\div(v) = R(v)+L(v)$.
    \end{itemize}    
    The \emph{genus} of the floor diagram $\D$ is the first Betti number of the underlying graph $\Gamma$. 
    We will often confuse $\D$ and $\Gamma$.
\end{defi}

\begin{rem}
    If $\Delta=\Delta_\beta$ for some $\beta \in H_2(X,\ZZ)$, a floor diagram of Newton polygon $\Delta_\beta$ is also said to have \emph{class} $\beta$. 
\end{rem}

Given a non-negative integer $n$, the quantum integer $[n]$ is the Laurent polynomial in $q^{1/2}$ defined by
\[ [n](q) = \dfrac{q^{n/2}-q^{-n/2}}{q^{1/2}-q^{-1/2}} \in \ZZ_{\geq0}[q^{\pm 1/2}] . \]

\begin{defi}
    Let $\D$ be a floor diagram. Its \emph{refined Block-Göttsche multiplicity} is
    \[  \mu_{BG}(\D) = \dprod_{e\in E(\D)} [w(e)]^2 . \]
    It is a symmetric Laurent polynomial in $q$.
\end{defi}

\begin{defi}
    Let $\D$ be a floor diagram of Newton polygon $\Delta$ and genus $g$. We define its \emph{degree} to be the degree of its multiplicity with cleared denominators, which is the Laurent polynomial $(q^{1/2}-q^{-1/2})^{b^\top+b^\bot+2|E^0(\D)|}\mu_{BG}(\D)$:
   \[ \deg(\D) = \frac{b^\top}{2}+\frac{b^\bot}{2}+\dsum_{e\in E^0(\D)} w(e), \]
    and its \emph{codegree} is the complement to the maximal degree of a genus $g$ floor diagram with Newton polygon $\Delta$, \ie to the euclidean area of $\Delta_\beta$:
    \[ \codeg(\D) =  \Area(\Delta_\beta)-\deg(\D). \]
\end{defi}

\begin{rem}
    From \cite{brugalle2008floor}, we know that floor diagrams actually encode simple tropical curves, which are dual to convex subdivisions of the Newton polygon $\Delta$ consisting of triangles and parallelograms. The codegree is actually the area of the parallelograms appearing in the subdivision.
\end{rem}

\begin{expl}
Consider the polygon $\Delta$ of figure \ref{fig-exemple-diag}. It has
\[ \gmax = 5,\ a = 3,\ b^\top = 1,\ b^\bot = 3,\ b_{\text{left}} = \{-1,1,1\} \text{ and } b_{\text{right}} = \{0,0,1\}. \]
The floor diagrams $\D_1,$ $\D_2$ and $\D_3$ have $\Delta$ as Newton polygon. We always represent floor diagrams oriented from bottom to top, hence we do not precise orientation on the figures. Besides, we indicate the weight of an edge only if it is at least 2. The genera of $\D_1,$ $\D_2$ and $\D_3$ are respectively 0, 1 and 1, and their codegrees are 0, 2 and 4. Their refined multiplicities are
\begin{align*}
    \mu_{BG}(\D_1) &= (q+1+q^{-1})^2 (q^{3/2}+q^{1/2}+q^{-1/2}+q^{-3/2})^2 \\
    &= q^5 +4q^4 + 10q^3 +18q^2 +25q + 28 + \dots \\
    \mu_{BG}(\D_2) &= (q^{1/2}+q^{-1/2})^2  (q^{1/2}+q^{-1/2})^2 \\
    &= q^2 +4q+6 +\dots \\
    \mu_{BG}(\D_3) &= 1.
\end{align*}

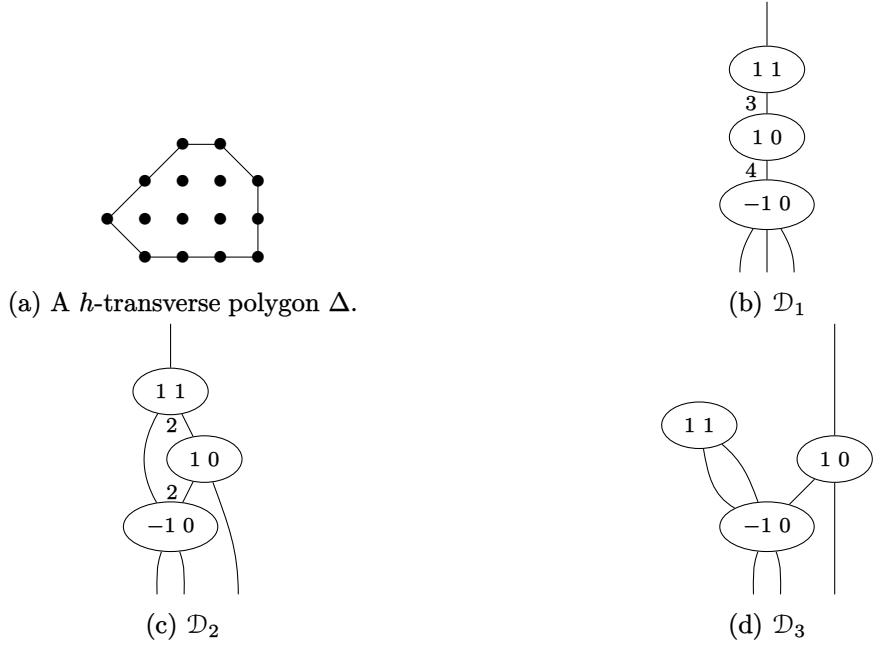
\begin{figure}[h!]
\begin{subfigure}[t]{0.49\textwidth}
    \centering
    \begin{tikzpicture}[scale=0.5] 
	\foreach \p in {(0,0), (1,0), (2,0), (3,0), (-1,1), (0,1), (1,1), (2,1), (3,1), (0,2), (1,2), (2,2), (3,2), (1,3), (2,3)}
			{\draw node at \p {$\bullet$} ;}
	\draw (0,0) to (3,0) to (3,2) to (2,3) to (1,3) to (-1,1) to cycle;
	\end{tikzpicture}
    \caption{A $h$-transverse polygon $\Delta$.}
\end{subfigure}
\begin{subfigure}[t]{0.49\textwidth}
\centering
    \begin{tikzpicture}[scale=0.9] 
    % sommets
        \ufloor (1) at (0,1) (\scriptsize $-1\ 0$);
        \ufloor (2) at (0,2) (\scriptsize $1\ 0$);
        \ufloor (3) at (0,3) (\scriptsize $1\ 1$);
    % arêtes
        \draw (1) to node[left] {\scriptsize $4$} (2) to node[left] {\scriptsize $3$}  (3) ;
		\draw (1) to[out=-120, in=90] (-0.4,0);
		\draw (1) to[out=-90, in=90] (0,0);
		\draw (1) to[out=-60, in=90] (0.4,0);
		\draw (3) to[out=90, in=-90] (0,4);
    \end{tikzpicture} 
    \caption{$\D_1$}
\end{subfigure}
    
\begin{subfigure}[t]{0.49\textwidth}
\centering
    \begin{tikzpicture}[scale=0.9] 
    % sommets
        \ufloor (1) at (0,1) (\scriptsize $-1\ 0$);
        \ufloor (2) at (0.5,2) (\scriptsize $1\ 0$);
        \ufloor (3) at (0,3) (\scriptsize $1\ 1$);
    % arêtes
        \draw (1) to node[left] {\scriptsize $2$} (2) to node[left] {\scriptsize $2$}  (3) ;
        \draw (1) to[out=120, in=-120] (3) ;
        
		\draw (1) to[out=-110, in=90] (-0.2,0);
		\draw (1) to[out=-70, in=90] (0.2,0);
		\draw (2) to[out=-70, in=90] (1,0);
		\draw (3) to[out=90, in=-90] (0,4);
	\end{tikzpicture}
    \caption{$\D_2$}
\end{subfigure}
\begin{subfigure}[t]{0.49\textwidth}
\centering
    \begin{tikzpicture}[scale=0.9] 
    % sommets
        \ufloor (1) at (0,1) (\scriptsize $-1\ 0$);
        \ufloor (2) at (-1,2.5) (\scriptsize $1\ 1$);
        \ufloor (3) at (1,2) (\scriptsize $1\ 0$);
    % arêtes
        \draw (1) to[out=110, in=-40] (2) ;
        \draw (1) to[out=150, in=-80] (2) ;
        \draw (1) to (3) ;
		\draw (1) to[out=-110, in=90] (-0.2,0);
		\draw (1) to[out=-70, in=90] (0.2,0);
		\draw (3) to[out=90, in=-90] (1,4);
		\draw (3) to[out=-90, in=90] (1,0);
	\end{tikzpicture}
    \caption{$\D_3$}
\end{subfigure}
\caption{Examples of floor diagrams.}
\label{fig-exemple-diag}
\end{figure}
\end{expl}

Notice that the orientation on a floor diagram $\D$ induces a partial order $\prec$ on $E(\D) \cup V(\D)$. We can thus define increasing functions on $E(\D) \cup V(\D)$ and the following definition makes sense.

\begin{defi}
    A \emph{marking} $\mfk$ of a floor diagram $\D$ is an increasing bijection $E(\D) \cup V(\D) \to \{1,\dots, |E(\D) \cup V(\D)|\}$. The pair $(\D,\mfk)$ is a \emph{marked floor diagram}. Two marked floor diagrams $(\D,\mfk)$ and $(\D',\mfk')$ are \emph{isomorphic} if there exists an isomorphism $\varphi : \D \to \D'$ of weighted graphs such that $L = L' \circ \varphi$, $R = R'\circ \varphi$ and $\mfk=\mfk' \circ \varphi$.  
\end{defi}

The following theorem can be taken as a definition of the tropical refined invariants.

\begin{theo}[{\cite[Theorem 4.3]{block2016fock}}]
\label{theo-tropical-refined-invariant}
    Let $X$ be a $h$-transverse toric surface, $\beta\in H_2(X,\ZZ)$, and $g\in\ZZ_{\geq0}$. The tropical refined invariant is given by
    \[ BG^X_g(\beta)(q) = \dsum_{(\D,\mfk)}\mu_{BG}(\D)(q) \in \ZZ[q^{\pm1}] \]
    where the sums runs over the isomorphism classes of marked floor diagrams with Newton polygon $\Delta_\beta$ and genus $g$. 
\end{theo}

The tropical refined invariant $BG^X_g(\beta)$ is a symmetric Laurent polynomial in $q$ with integer coefficients. Its degree is $|\mathring\Delta_\beta \cap \ZZ^2|-g = \gmax(\Delta_\beta)-g$.

\begin{lem}\label{lem-poids-grands}
    Let $g\geq 0$ and $M,i>0$. Let $\Delta$ be a $h$-transverse polygon. Assume
    $$b^\top(\Delta), b^\bot(\Delta) > M(g+1)+i.$$
    Let $\D$ be a floor diagram of genus $g$ and $\codeg(\D) \leq i$. Then for any consecutive floors $v_m$ and $v_{m+1}$, there is a bounded edge $e$ between them with weight $w(e) > M$.
\end{lem}

\begin{proof}
 
    Let $\omega_m$ be the weight of the unique edge between the floors $m$ and $m+1$ in the unique floor diagram of genus 0 and codegree 0 with Newton polygon $\Delta$. It is equal to the integral length of the slice $\Delta\cap(\ZZ\times\{m\})$ of $\Delta$ at height $m$. As $\Delta$ is convex, it is always bigger than $\min(b^\top,b^\bot)$.
    
    Let $\D$ be a genus $g$ marked diagram, and let $v_1, \dots, v_a$ be the floors of $\D$, ordered by their marking. Let $\tomega_m$ be the sum of the weights of the bounded edges of $\D$ that link two floors $v_m$ and $v_{m+1}$. It may be $0$ if there are no such edges. In comparison to a diagram of genus $g$ and codegree 0, the codegree of $\D$ comes from two different phenomena: 
    \begin{itemize}
        \item an edge with weight $w$ that skips $k$ floors including $v_m$ or $v_{m+1}$ contributes $kw$ to the codegree, and decreases by $w$ the maximal value of weight $\tomega_m$ between $v_m$ and $v_{m+1}$,

        \item two floors $v_k\prec v_{k+1}$ having $R(v_k) > R(v_{k+1})$ or $L(v_k) > L(v_{k+1})$ contribute at least one to the codegree, and decrease the total weight of the bounded edges between $v_k$ and $v_{k+1}$ by at least $1$.
    \end{itemize}
Hence there are at most $\codeg(\D)$ such phenomena, and therefore
    $$\tomega_m \geq \omega_m-\codeg(\D) \geq \omega_m-i>(g+1)M.$$
    In particular, there is at least one edge between the floors $v_m$ and $v_{m+1}$, so that the floors are totally ordered in the diagram. Because $\D$ has genus $g$, the total weight $\tilde\omega_m$ is split into at most $g+1$ edges. Hence we have at least one of them with $w(e) > M$.
\end{proof}

\subsection{Asymptotic refined invariants}
\label{sec-asymp-refined-inv}

In several situations, it is interesting to remove the denominators from the refined Block-G\"ottsche multiplicities, as in \cite{mikhalkin2017quantum} and \cite{bousseau2019tropical} for instance. In our situation, results also adopt a simpler form if we do so. We will thus forget about the denominators. Then $BG^X_g(\beta)(q)$ becomes a Laurent polynomial obtained by counting marked diagrams with a different multiplicity:
\[ \bar{BG}^X_g(\beta)(q) = \dsum_{(\D,\mfk)}\ \dprod_{e\in E^0(\D)} (q^{w(e)/2}-q^{-w(e)/2})^2 \dprod_{e\in E^\top(\D) \cup E^\bot(\D)}(q^{1/2}-q^{-1/2}). \]
The first product is obtained clearing the denominators of each $[w]^2$. The second product comes from the ends: actually, each end contributes $[1]=\frac{q^{1/2}-q^{-1/2}}{q^{1/2}-q^{-1/2}}=1$ in $BG^X_g(\beta)$, and this becomes $q^{1/2}-q^{-1/2}$ when clearing denominators. 

The degree of the invariant $\bar{BG}^X_g(\beta)$ with cleared denominators is $\mathrm{Area}(\Delta_\beta)$. This Laurent polynomial is symmetric (resp. antisymmetric) when $\beta\cdot K_X$ is even (resp. odd).

\begin{expl}
For the toric surface $\PP^2$, if $L$ is the class of a toric divisor one has
\begin{align*}
\bar{BG}^{\PP^2}_0(L)(q)= & q^{1/2}-q^{-1/2}, \\ 
\bar{BG}^{\PP^2}_0(2L)(q)= & (q^{1/2}-q^{-1/2})^4\cdot 1 \\
= & q^2-4q+6-4q^{-1}+q^{-2},\\
\bar{BG}^{\PP^2}_0(3L)(q)= & (q^{1/2}-q^{-1/2})^7\cdot (q+10+q^{-1}) \\
= & q^{9/2}+3 q^{7/2} -48 q^{5/2} +168q^{3/2} -294q^{1/2}+\cdots -q^{-9/2}.
\end{align*}
\end{expl}

The Laurent polynomial $\bar{BG}^X_g(\beta)(q)$ can be turned into a true polynomial by setting
\[\widetilde{BG}^X_g(\beta)(x)=x^{\mathrm{Area}(\Delta_\beta)}\bar{BG}^X_g(\beta)\left(\frac{1}{x}\right).\]

This way, the codegree $i$ coefficient of $\bar{BG}^X_g(\beta)$ (term $q^{\mathrm{Area}(\Delta_\beta)-i}$) becomes the degree $i$ coefficient of $\widetilde{BG}^X_g(\beta)$ (term $x^i$). Thanks to this change of variable, we can now view the refined invariant as a function
\[\begin{array}{crcl}
   \widetilde{BG}^X_g : &   \Eff(X) & \longrightarrow & \ZZ[\![x]\!], \\
     & \beta & \longmapsto & \widetilde{BG}^X_g(\beta)(x)
\end{array}  \]
with values in the ring of formal series with integer coefficients $\ZZ[\![x]\!]$, even though for any $\beta$ the value is polynomial in $x$ of degree $2\mathrm{Area}(\Delta_\beta)$. The codomain $\ZZ[\![x]\!]$ is a valuation ring, and can thus be endowed with the topology coming from the associated ultrametric distance. A basis of neighborhoods of $0$ for this topology is given by the ideals $x^n\ZZ[\![x]\!]$, so that $f\in\ZZ[\![x]\!]$ is close to $0$ if $f\equiv 0 \modulo x^n$ for $n\in\NN$ sufficiently big. We prefer to use $\ZZ[\![x]\!]$ as codomain, because it is a complete space, more suited to express our asymptotic result.  Meanwhile, we have a notion of neighborhood of infinity in the cone $\Eff(X)$: for $C>0$, we say that $\beta\succ C$ if $\beta\cdot D\geqslant C$ for every toric divisor $D$.

\begin{prop}\label{prop-new-mult}
    For any $h$-transverse toric surface $X$, any $\beta \in H_2(X,\ZZ)$ and any $g\geq0$, one has 
    \[ \widetilde{BG}^X_g(\beta)(x) = \dsum_{(\D,\mfk)} \mu(\D)(x) \]
    where the sum runs over the isomorphism classes of marked floor diagrams of Newton polygon $\Delta_\beta$ and genus $g$, and with
    \[ \mu(\D)(x) = x^{\codeg(\D)} (1-x)^{b^\top(\D)+b^\bot(\D)} \dprod_{e\in E^0(\D)} \left(1-x^{w(e)} \right)^2.\]
    We call $\mu(\D)$ the \emph{multiplicity} of the floor diagram $\D$.
\end{prop}

\begin{proof}
By the definition of the codegree one has
     \[ \Area(\Delta_\beta) = \codeg(\D) + \dsum_{e\in E^0(\D)} w(e) + \dfrac{b^\top+b^\bot}{2} .\]
    Hence
    \begin{align*}
    \widetilde{BG}^X_g(\beta) &=  x^{\mathrm{Area}(\Delta_\beta)}  \dsum_{(\D,\mfk)}\ \dprod_{e\in E^0(\D)} (x^{-w(e)/2}-x^{w(e)/2})^2 \dprod_{\substack{e\in E^\top(\D) \\ \cup E^\bot(\D)}}(x^{-1/2}-x^{1/2}) \\
        &= \dsum_{(\D,\mfk)} x^{\codeg(\D)} \dprod_{e\in E^0(\D)} (1-x^{w(e)})^2  \dprod_{\substack{e\in E^\top(\D) \\ \cup E^\bot(\D)}} (1-x) \\
        &= \dsum_{(\D,\mfk)} \mu(\D).
    \end{align*}
\end{proof}

The main results of \cite{brugalle2020polynomiality} and \cite{mevel2023universal} deal with the polynomiality in terms of $\beta$ of the coefficient of fixed codegree $i$ of $BG^X_g(\beta)(q)$, which become the $x^i$ coefficient of $\widetilde{BG}_g^X(\beta)(x)$ in our setting. The polynomial behaviour is not affected by the clearing of denominators. Indeed, invariants with or without clearing of denominators differ by multiplication (or division) by $(1-x)^{-K_X\cdot\beta+2g-2}$, whose coefficients up to degree $i$ are also polynomials in $\beta$. We denote the degree $i$ coefficient by $\left(\widetilde{BG}^X_g(\beta)\right)_i$. In \cite{brugalle2020polynomiality} the authors show that when $X$ is a Hirzebruch surface or a (weighted) projective space, then for any $i\geq 0$ there exists a polynomial function $AR^X_{g,i}(\beta)$ on the lattice $H_2(X,\ZZ)$ such that for every $\beta$ with $\beta\cdot D$ large enough with respect to $i$ and $g$ for any $D$ toric divisor, one has  
\[ \left(\widetilde{BG}^X_g(\beta)(x)\right)_i=AR^X_{g,i}(\beta).\]
Taking the generating series in $i$, the result can be rephrased as follows.

\begin{theo}(\cite{brugalle2020polynomiality})
    Let $X$ be a Hirzebruch surface. For every $g\geqslant 0$, there exists a function $AR^X_g:H_2(X,\ZZ)\to\ZZ[\![x]\!]$ which is a polynomial with coefficients in $\ZZ[\![x]\!]$, such that
    $$\widetilde{BG}^X_g(\beta)=AR^X_g(\beta)+o(1)\in\ZZ[\![x]\!].$$
    The asymptotic development takes place when $\beta\to\infty$. We call $AR^X_g$ the \emph{asymptotic refined invariant}.
\end{theo}

\begin{rem}
The result is likely to be true for any toric surface, but in \cite{brugalle2020polynomiality} the authors restrict the proof to a family of surfaces that includes Hirzebruch surfaces for technical reasons. In genus 0, this result is shown to hold in \cite{mevel2023universal} for any $h$-transverse toric surface, with explicit formula for the polynomials. 
We give in this paper another proof in genus 0 (Theorem \ref{theo-AR-genus0-general}) and a proof in genus 1 (Theorem \ref{theo-AR-genus1-general}) that holds for any $h$-transverse and non-singular toric surface.
\end{rem}

\begin{proof}[Proof of the formulation using \cite{brugalle2020polynomiality}.]
    The formulation of the result presented here amounts to prove that there exists polynomials $P^X_{g,i}(\beta)$ whose degree is bounded by a constant in $g$ such that we have 
    $$\widetilde{BG}^X_g(\beta)=\sum_{i=0}^\infty P^X_{g,i}(\beta)x^i+o(1).$$
    The $o(1)$ means that for every $n\in\NN$, there exists $C>0$ such that
    $$\beta\succ C\Rightarrow \widetilde{BG}^X_g(\beta)-\sum_{i=0}^\infty P^X_{g,i}(\beta)x^i \equiv 0 \modulo x^n.$$
    In other words, $\langle\widetilde{BG}^X_g(\beta)\rangle_i$ is given by $P^X_{g,i}$ if $\beta$ is sufficiently big. This is the statement from \cite{brugalle2020polynomiality} up to the bound on the degrees. Actually, in \cite{brugalle2020polynomiality} the degree of $P^X_{g,i}$ is $g+i$. The dependence in $i$ is due to the fact that the denominators are not removed from the Block-G\"ottsche multiplicity. If we remove them, \cite[Section 3.2]{brugalle2020polynomiality} is modified as follows. The function $\Phi_i(k)$ in \cite[Corollary 3.6]{brugalle2020polynomiality} does not depend on $k$ anymore: it was previously a polynomial of degree $i$ but is now constant, equal to $1$ if $i=0$, and to $0$ if $i>0$. Thus, in \cite[Corollary 3.7]{brugalle2020polynomiality}, the degree is also $0$. When used in the proof of \cite[Lemma 5.7]{brugalle2007enumeration}, the bound $i$ disappears and yields the fact that the degrees are bounded by $g$.
\end{proof}

One way to interpret the asymptotic is that for any $i$ and any class $\beta$ large enough (understand $\beta\succ C$ for $C\in\NN$ large enough), $AR^X_g(\beta)$ correctly gives the first $i$ coefficients of $\widetilde{BG}^X_g(\beta)$. The strategy to compute the asymptotic refined invariant $AR_g^X$ is thus to fix some $i\in\NN$ and to compute $\widetilde{BG}^X_g(\beta)$ modulo $x^{i+1}$ for $\beta$ big enough. Provided we get an expression that does not depend on $i$, we can make $i$ go to $\infty$ to obtain the value of $AR_g^X(\beta)\in\ZZ[\![x]\!]$. 

\medskip

This formulation as an asymptotic development is in fact inspired by a reformulation of the polynomiality conjecture \cite{gottsche1998conjectural} on the number of curves with a fixed number of nodes $\delta$ passing through a suitable number of points.
For a surface $X$ with a line bundle $\L$, G\"ottsche's conjecture states that the number $N_\delta^X(\L)$ of curves in the linear system $|\L|$ with $\delta$ nodes passing through a generic configuration of $h^0(X,\L)-1-\delta$ points in $X$ is given by a (universal) polynomial $P^X_\delta(\L^2,\L\cdot K_X,K_X^2,c_2(X))$ provided $\L$ is sufficiently ample. In other words, if we view $P^X_\delta$ and $N^X_\delta$ as functions
$$P^X_\delta,N^X_\delta:\mathrm{Amp}(X)\subset H^2(X,\ZZ)\longrightarrow \ZZ,$$
where $\mathrm{Amp}(X)\subset H^2(X,\ZZ)$ is the ample cone of $X$, we have
\[ N^\delta(\L)=P^\delta(\L)+o(1). \]
As the topology of $\ZZ$ is discrete, it means that we have equality if $\L$ is sufficiently ample. The formulation in the refined case is more subtle since the topology on $\ZZ[\![x]\!]$ is more sophisticated.

\section{Generating series in fixed degree}
\label{sec-series-fixed-codeg}

In this section, we wish to determine the generating series in the genus parameter, \ie $\sum_{g=0}^\infty AR_{g,i} u^g$ for fixed $i$. We provide an explicit expression for $i=0,1$. The $i=0$ case amounts to compute the leading coefficient of the tropical refined invariant, which was already known from \cite{itenberg2013block}. The main contribution is Theorem \ref{theo-ARcodeg1}, which gives a closed formula for $i=1$.

Recall from \cite{brugalle2020polynomiality} or \cite{mevel2023universal} that a diagram $\D$ has codegree 0 if and only if the order is total on its floors, it has no side edge (\ie an edge bypassing a floor), and the functions $R$ and $L$ are increasing. Recall also that that when looking at floor diagrams of small codegree, we can assume using Lemma \ref{lem-poids-grands} that the diagrams have a total order on their vertices, and we can control the number of side edges as well as the monotonicity of the functions $L$ and $R$.

In all this section, for $\beta \in H_2(X,\ZZ)$ with $\Delta_\beta$ being $h$-transverse, we will refer as $\D_0$ to be the floor diagram of Figure \ref{fig-diag00}. It is the unique diagram of Newton polygon $\Delta_\beta$, genus 0 and codegree 0. We denote by $\omega_m$ the weight of the edge between the floors $v_m$ and $v_{m+1}$ for $1\leq m \leq a-1$. Note that
\[ \dsum_{m=1}^{a-1} (\omega_m-1) = \gmax  = \deg(\D_0) . \]
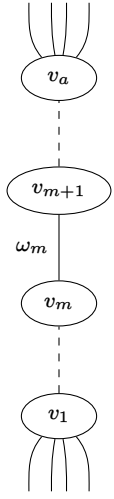
\begin{figure}[h!]
    \centering
    \begin{tikzpicture}%[line cap=round,line join=round,x=0.75cm,y=0.75cm]
        % sommets
        \ufloor (1) at (0,1) ($v_1$);
        \ufloor (2) at (0,2.5) ($v_m$);
        \ufloor (3) at (0,4) ($v_{m+1}$);
        \ufloor (4) at (0,5.5) ($v_a$);
        % arêtes
		\draw (1) to[out=-130, in=90] (-0.4,0);
		\draw (1) to[out=-100, in=90] (-0.1,0);
		\draw (1) to[out=-80, in=90] (0.1,0);
		\draw (1) to[out=-50, in=90] (0.4,0);

		\draw (4) to[out=130, in=-90] (-0.4,6.5);
		\draw (4) to[out=100, in=-90] (-0.1,6.5);
		\draw (4) to[out=80, in=-90] (0.1,6.5);
		\draw (4) to[out=50, in=-90] (0.4,6.5);

        \draw[dashed] (1) to (2);
        \draw (2) to node[left] {\scriptsize $\omega_m$} (3) ;
        \draw[dashed] (3) to (4);

    \end{tikzpicture} 
    \caption{The diagram $\D_0$.}
    \label{fig-diag00}
\end{figure}

\subsection{Degree 0} 

We start by computing $AR^X_{g,0}$, the leading term of the asymptotic refined invariant. This amounts to compute the leading coefficient of the tropical refined invariant, which was already handled in \cite[Proposition 2.11]{itenberg2013block} using the lattice path algorithm from \cite{mikhalkin2005enumerative}. We recall a proof here, because it uses a construction starting from $\D_0$ that will appear several times in subsection \ref{subsec-deg1}.

\begin{prop}[{\cite{itenberg2013block}}]\label{theo-ARcodeg0}
The generating series in the genus parameter of the leading term of the asymptotic refined invariant is given by
\[ \dsum_{g\geq0} AR_{g,0}^X u ^g = (1+u)^{\gmax} .\]
\end{prop}

\begin{proof}
To construct a marked floor diagram of positive genus and codegree 0, we add $g_m$ edges between the floors $v_m$ and $v_{m+1}$ of $\D_0$, marking the new edges increasingly from left to right, and splitting the weight $\omega_m$ onto the $g_m+1$ edges. The genus of the new diagram is $g_1+\dots+g_{a-1}$. For each $m$ there are $\binom{\omega_m-1}{g_m}$ tuples of $g_m+1$ positive integers with sum $\omega_m$, \ie ways to distribute $\omega_m$ onto the marked edges. Since we only care about the number of marked diagrams of genus $g$ to compute $AR_{g,0}^X$, using the binomial formula one has 
\begin{align*}
\dsum_{g\geq0} AR_{g,0}^X u ^g &=  \dsum_{g_1,\dots,g_{a-1} \geq 0} \ u^{g_1+\dots+g_{a-1}} \dprod_{m=1}^{a-1} \binom{\omega_m-1}{g_m} \\ 
    &= \dprod_{m=1}^{a-1} \dsum_{g_m\geq0}  \binom{\omega_m-1}{g_m} u^{g_m} \\
    &= \dprod_{m=1}^{a-1} (1+u)^{\omega_m-1} = (1+u)^{\gmax}.
\end{align*}
\end{proof}

\subsection{Degree 1} \label{subsec-deg1}

We now compute the generating series of the second terms of the asymptotic refined invariants.

\begin{theo}\label{theo-ARcodeg1}
For a $h$-transverse toric surface $X$, the asymptotic polynomials yielding the degree $1$ coefficient are polynomials in $\beta^2,K_X\cdot\beta,\chi,K_X^2$. Moreover, their generating series has the following expression:
\[\dsum_{g\geq 0} AR_{g,1}^X u^g  = (1+u)^{\gmax} \left[-\beta^2 \dfrac{u^3}{(1+u)^3} +2(K_X\cdot\beta)\dfrac{u^2}{(1+u)^3} + \chi\dfrac{1}{1+u} - K_X^2 \dfrac{u}{(1+u)^3} \right] . \]

\end{theo}

Given that the multiplicity takes the form
\[ \mu(\D) = x^{\codeg(\D)} (1-x)^{b^\top+b^\bot} \dprod_{e\in E_0(\D)} \left(1-x^{w_e}\right)^2, \]
only diagrams with $\codeg(\D)=0,1$ contribute to $\sum_g AR_{g,1}^X u^g$. For the unique marked diagram of codegree $0$ we need to consider the term in $x$, while for the marked floor diagrams of codegree $1$ we need to consider their number. We subdivide the proof of Theorem \ref{theo-ARcodeg1} in four lemmas, each one computing the contribution of a specific family of diagrams to the global sum. To get Theorem \ref{theo-ARcodeg1}, one only needs to sum the expressions from Lemmas \ref{lem-contrib-codeg0}, \ref{lem-contrib-codeg1-iside}, \ref{lem-contrib-codeg1-bside} and \ref{lem-contrib-codeg1-corner}.

\begin{lem}\label{lem-contrib-codeg0}
    The codegree 0 diagrams with Newton polygon $\Delta_\beta$ contribute 
    \[  (1+u)^{\gmax} \left[ -(b^\top+b^\bot) -2\gmax\frac{u^2}{(1+u)^2}-2(a-1)\frac{u(2+u)}{(1+u)^2} \right] \]
    to $\sum_g AR_{g,1}^X u^g$.
\end{lem}

\begin{proof}
We construct a diagram of genus $g$ and codegree 0 as in the proof of Proposition \ref{theo-ARcodeg0}. A diagram $\D$ of codegree 0 is counted with the degree 1 term of its multiplicity, that is
\[ -(b^\top+b^\bot) - 2 |\{e\in E^0(\D)\ |\ w_e=1\}| .\]
Hence, the contribution to $\sum_g AR_{g,1}^X u^g$ coming from the first term is $-(b^\top+b^\bot) (1+u)^{\gmax}$. We need to compute the contribution coming from the second term, \ie enumerate the choice of a diagram together with an edge of weight $1$. To determine this contribution, we proceed as previously but for any fixed $m$, we assume one of the $g_m+1$ edges between $v_m$ and $v_{m+1}$ has weight $1$, and it remains a weight $\omega_m-1$ to split into $g_m$ parts. Forgetting the $-2$, this gives
\[ \begin{array}{rl}
 & \dsum_{m=1}^{a-1} \left(\sum_{\substack{g_j\geqslant 0 \\j\neq m}}\prod_{j\neq m}\binom{\omega_j-1}{g_j}u^{g_j}\right)\left(\sum_{g_m\geqslant 0} \dsum_{\substack{\text{edges} \\ v_m\to v_{m+1}}} \binom{\omega_m-2}{g_m-1}u^{g_m} \right) \\
 = & \dsum_{m=1}^{a-1} (1+u)^{\sum_{j\neq m}[\omega_j-1]}\sum_{g_m\geqslant 0} (g_m+1) \binom{\omega_m-2}{g_m-1}u^{g_m} \\
 = & \dsum_{m=1}^{a-1} (1+u)^{\gmax-(\omega_m-1)}\left[ (\omega_m-2)u^2(1+u)^{\omega_m-3}+2u(1+u)^{\omega_m-2}  \right] \\
 = & (1+u)^{\gmax}\left[ \gmax\dfrac{u^2}{(1+u)^2}+(a-1)\dfrac{u(2+u)}{(1+u)^2} \right].
\end{array} \] 
\end{proof}

We now look at the diagrams of codegree 1. The degree 1 term of the multiplicity of a diagram of codegree 1 is 1, so it suffices to determine the number of marked floor diagrams of codegree 1. There are two possibilities for the codegree being 1: the presence of a side edge, \ie an edge bypassing a floor, or a slope inversion, \ie a lack of growth of the divergence function. We investigate all the cases.

\begin{lem} \label{lem-contrib-codeg1-iside}
    The codegree 1 diagrams with an infinite side edge, with Newton polygon $\Delta_\beta$ contribute 
    \[ (1+u)^{\gmax}\left[ (\omega_1+\omega_{a-1}-2)\frac{u}{(1+u)^2} + (b^\bot+b^\top)\frac{1}{1+u}  +2\frac{2+u}{(1+u)^2}\right] \]
    to $\sum_g AR_{g,1}^X u^g$.
\end{lem}

\begin{proof}
We deal with the case when the side edge is a source ; the case when it is a sink is handled similarly by symmetry.

Let $\D_\bot$ be the diagram of Figure \ref{fig-sidesource} ; it is obtained from $\D_0$ by putting a source adjacent to $v_2$. It has genus 0 and codegree 1. Let $\tomega_k$ be the weight of the edge between $v_k$ and $v_{k+1}$ for $1\leq k \leq a-1$. One has
\[  \tomega_1 = \omega_1 -1 \text{ and } \tomega_k = \omega_k,\ 2 \leq k \leq a-1.
\]
To create a diagram of genus $g$, as in Theorem \ref{theo-ARcodeg0} we add $g_m$ edges between the floor $v_m$ and $v_{m+1}$ of $\D_\bot$, marking the new edges increasingly from left to right, and split the weight $w_m$ onto the $g_m+1$ edges. The genus of the new diagram is $g_1+\dots+g_{a-1}$ and for each $m$ there are $\binom{\tomega_m-1}{g_m}$ ways to distribute $\tomega_m$ onto the marked edges. To entirely determine the marked floor diagram, it remains to mark the side edge. It is parallel to $(g_1+1) + (b^\bot -1)$ edges and 1 floor, hence there are $g_1 + b^\bot+2$ possibilities for its marking. In the end, this case contributes
\[ \begin{array}{rl}
     & \dsum_{g_1,\dots,g_{a-1} \geq 0} \ (g_1 + b^\bot+2) \dprod_{m=1}^{a-1} \binom{\tomega_m-1}{g_m} u^{g_m} \\
    =& (1+u)^{\gmax-(\omega_1-1)} \dsum_{g_1 \geq0}  (g_1 + b^\bot+2) \binom{\tomega_1-1}{g_1}u^{g_1}  \\
    =& (1+u)^{\gmax-(\omega_1-1)} \left[ (\tomega_1-1) u (1+u)^{\tomega_1-2} + (b^\bot+2)(1+u)^{\tomega_1-1} \right] \\
    =& (1+u)^{\gmax}\left[ (\omega_1-1)\dfrac{u}{(1+u)^2} + b^\bot\dfrac{1}{1+u}+\dfrac{2+u}{(1+u)^2} \right]
    \end{array} \]

Similarly, if the side edge is a sink we get
\[ (1+u)^{\gmax}\left[ (\omega_{a-1}-1)\frac{u}{(1+u)^2} + b^\top\frac{1}{1+u} +\frac{2+u}{(1+u)^2} \right].\]
We get the result summing the two cases.
\end{proof}

\begin{figure}[h!]
\begin{subfigure}[t]{0.49\textwidth}
    \centering
    \begin{tikzpicture}%[line cap=round,line join=round,x=0.75cm,y=0.75cm]
        % sommets
        \ufloor (1) at (0,1.5) ($v_1$);
        \ufloor (2) at (0,3) ($v_2$);
        \ufloor (3) at (0,5) ($v_a$);
        % arêtes
		\draw (1) to[out=-130, in=90] (-0.6,0);
		\draw (1) to[out=-100, in=90] (-0.2,0);
		\draw (1) to[out=-80, in=90] (0.2,0);
		\draw (1) to[out=-50, in=90] (0.6,0);

		\draw (3) to[out=130, in=-90] (-0.6,6.5);
		\draw (3) to[out=100, in=-90] (-0.2,6.5);
		\draw (3) to[out=80, in=-90] (0.2,6.5);
		\draw (3) to[out=50, in=-90] (0.6,6.5);

        \draw (1) to node[left] {\scriptsize $\tomega_1$} (2);
        \draw[dashed] (2) to (3);

        \draw (2) to[out=-50,in=90] (1,0) ;
    \end{tikzpicture} 
    \caption{The diagram $\D_\bot$.}
    \label{fig-sidesource}
\end{subfigure}
\begin{subfigure}[t]{0.49\textwidth}
    \centering
    \begin{tikzpicture}%[line cap=round,line join=round,x=0.75cm,y=0.75cm]
        % sommets
        \ufloor (1) at (0,1) ($v_1$);
        \ufloor (2) at (0,2.5) ($v_j$);
        \ufloor (3) at (0,3.5) ($v_{j+1}$);
        \ufloor (4) at (0,4.5) ($v_{j+2}$);
        \ufloor (5) at (0,6) ($v_a$);
        % arêtes
		\draw (1) to[out=-130, in=90] (-0.4,0);
		\draw (1) to[out=-100, in=90] (-0.1,0);
		\draw (1) to[out=-80, in=90] (0.1,0);
		\draw (1) to[out=-50, in=90] (0.4,0);

		\draw (5) to[out=130, in=-90] (-0.4,7);
		\draw (5) to[out=100, in=-90] (-0.1,7);
		\draw (5) to[out=80, in=-90] (0.1,7);
		\draw (5) to[out=50, in=-90] (0.4,7);

        \draw[dashed] (1) to (2);
        \draw (2) to node[left] {\scriptsize $\tomega_j$} (3) to node[left] {\scriptsize $\tomega_{j+1}$} (4);
        \draw[dashed] (4) to (5);

        \draw (2) to[out=30,in=-30] (4) ;
    \end{tikzpicture} 
    \caption{The diagram $\D_j$.}
    \label{fig-sidebounded}
\end{subfigure}
\caption{ }
\end{figure}

\begin{lem}\label{lem-contrib-codeg1-bside}
    The codegree 1 diagrams with a bounded side edge contribute 
    \[  (1+u)^{\gmax} \left[ \dsum_{j=1}^{a-2}\left( \omega_j+\omega_{j+1}-2\right)\frac{u^2}{(1+u)^3} + 2(a-2) \frac{u(2+u)}{(1+u)^3} \right]  \]
    to $\sum_g AR_{g,1}^X u^g$.
\end{lem}

\begin{proof}
Start with the diagram $\D_j$ of figure \ref{fig-sidebounded} ; it has genus 1 and a side edge around the floor $v_{j+1}$. Let $\tomega_m$ be the weight of the edge between $v_m$ and $v_{m+1}$ for $1\leq m \leq a-1$. One has
\[ \tomega_j = \omega_j -1,\ \tomega_{j+1} = \omega_{j+1}-1 \text{ and } \tomega_m = \omega_m,\ m\notin \{j,j+1\}.\]
As previously, we add $g_m$ edges between the floor $v_m$ and $v_{m+1}$ of $\D_\bot$, mark the new edges increasingly from left to right, and split the weight $w_m$ onto the $g_m+1$ edges. The created diagram as genus $1+g_1+\dots+g_{a-1}$. The side edge is parallel to 1 floor and $g_j+g_{j+1}+2$ edges, so there are $g_j+g_{j+1}+4$ possibilities for its marking. Hence, the contribution in that case is
\[ \begin{array}{rl}
     & u \dsum_{j=1}^{a-2} \ \dsum_{g_1,\dots,g_{a-1} \geq 0} (g_j+g_{j+1}+4) \dprod_{m=1}^{a-1} \binom{\tomega_m-1}{g_m} u^{g_m} \\
     =&  \dsum_{j=1}^{a-2} (\tomega_j+\tomega_{j+1}-2)u^2 (1+u)^{\gmax-3} + 4(a-2)u (1+u)^{\gmax-2} \\
    =& (1+u)^{\gmax} \left[ \dsum_{j=1}^{a-2} (\omega_j+\omega_{j+1}-2)\dfrac{u^2}{(1+u)^3} + 2(a-2) \dfrac{u(2+u)}{(1+u)^3} \right] .
\end{array} \]
\end{proof}

\begin{lem}\label{lem-contrib-codeg1-corner}
    The codegree 1 diagrams with an slope inversion contribute 
    \[ (\chi-4) (1+u)^{\gmax-1}  \]
    to $\sum_g AR_{g,1}^X u^g$.
\end{lem}

\begin{proof}
To get a floor diagram of codegree 1 with an inversion, the only possibility is the existence of a unique couple $(v,v')$ of adjacent floors such that $v \prec v'$ and $R(v) = R(v')+1$ or $L(v) = L(v')+1$, and anywhere else in the floor diagram, $R$ and $L$ are increasing. If $\chi$ is the number of corners of $\Delta_\beta$, there are $\chi-4$ such pairs, one for each corner of $\Delta$ non-adjacent to a horizontal side. The only difference with the codegree $0$ diagram from Figure \ref{fig-diag00} is that the weight between $v_m$ and $v_{m+1}$ is $\omega_m-1$ so that the sum of weights yields $\gmax-1$ instead of $\gmax$. In the end, this case contributes
\[ (\chi-4) (1+u)^{\gmax-1} .\]
\end{proof}

We can finally prove Theorem \ref{theo-ARcodeg1}.

\begin{proof}[Proof of Theorem \ref{theo-ARcodeg1}]
Summing the contributions of Lemmas \ref{lem-contrib-codeg0}, \ref{lem-contrib-codeg1-iside}, \ref{lem-contrib-codeg1-bside} and \ref{lem-contrib-codeg1-corner}, and using the relations
\[ \left\{\begin{array}{rcl}
    \gmax &=& \dsum_{m=1}^{a-1}(\omega_m-1), \\
     -K_X\cdot\beta &=& b^\top+b^\bot +2a , \\
     \beta^2 &=& 2\gmax-2+K_X\cdot\beta  \\
     K_X^2 + \chi &=& 12 
\end{array} \right. \]
we get
\[\dsum_{g\geq 0} AR_{g,1}^X u^g  = (1+u)^{\gmax} \left[-\beta^2 \dfrac{u^3}{(1+u)^3} +2(K_X\cdot\beta)\dfrac{u^2}{(1+u)^3} + \chi\dfrac{1}{1+u} - K_X^2 \dfrac{u}{(1+u)^3} \right] . \]
\end{proof}

\section{Asymptotic refined invariant for genus $0$}
\label{sec-genus0}

In this section we compute the asymptotic refined invariant for genus $0$ for any $h$-transverse and non-singular polygon having two horizontal sides. This was already done in \cite{mevel2023universal}, but we give in this section a different proof to present methods that can be applied when dealing with genus 1 in Section \ref{sec-genus1}. We start with Hirzebruch surfaces before going into the general case. To do so, we use \emph{words} to enumerate marked floor diagrams contributing to the asymptotic count. We will compute $\widetilde{BG}_0^X(\beta)$ modulo $x^{i+1}$ for some $i$, before letting $i$ goes to $\infty$.

    \subsection{The case of Hirzebruch surfaces}

The tropical refined invariants can be computed by using enumeration of marked floor diagrams. However, as shown in \cite[Lemma 4.1]{brugalle2020polynomiality}, if one cares about the asymptotic of coefficients of fixed codegree only a handful of diagrams contribute. Consider the Hirzebruch surface $\FF_\delta$, so that all floors have the same divergence. In the genus $0$ case, provided that $a,b>i$, any marked diagram contributing to a coefficient of degree at most $i$ satisfies the following:
    \begin{itemize}
        \item the floors are totally ordered in the diagram,
        \item some of the top (resp. bottom) ends might not be attached to the first (resp. last) floor but to another floor,
    \end{itemize}
Let $u^\top_j$ (resp. $u^\bot_j$) be the number of top (resp. bottom) ends that skip $j$ floors. The codegree of a diagram $\D$ comes from these ends not attached to the extremal floors. It is equal to
$$\codeg(\D)=\sum_{j=1}^\infty j(u^\top_j+u^\bot_j),$$
note that this sum is actually finite. 
Each diagram is characterized by the numbers $(u_j^\top,u_j^\bot)$. We then have to account for the markings. We restrict to the collection of diagrams for which the floors are totally ordered. Rather than enumerating the latter and count their markings, as done in \cite[Section 4]{brugalle2020polynomiality} and \cite{mevel2023universal}, we directly count these marked marked diagrams, encoding them with \emph{words}.

\medskip

\subsubsection{From marked diagrams to words.}

We consider words over the following alphabet : $\{\sff,\sfe,\sfb_j,\sft_j\}_{j\in\NN}$. The letters used stand for ``floor'', ``edge/elevator'', ``bottom end'' and ``top end''. The indices of the letters refer to the number of floors they skip. We first explain how to get a word $W(\D)$ from a genus $0$ marked diagram $\D$ whose floors are totally ordered. Let $aE+bF\in H_2(\FF_\delta,\ZZ)$ be the class of the diagram. The floors of $\D$ are labelled from $1$ to $a$. The letters of the word $W(\D)$ are in ordered correspondence with the marked points of $\D$ with the following rule:
    \begin{itemize}
        \item for a marked point on a floor, the letter is $\sff$,
        \item for a marked point on a bounded edge, the letter is $\sfe$,
        \item for a marked point on a top end that skips $j\geqslant 0$ floors, the letter is $\sft_j$,
        \item for a marked point on a bottom end that skips $j\geqslant 0$ floors, the letter is $\sfb_j$.
    \end{itemize}

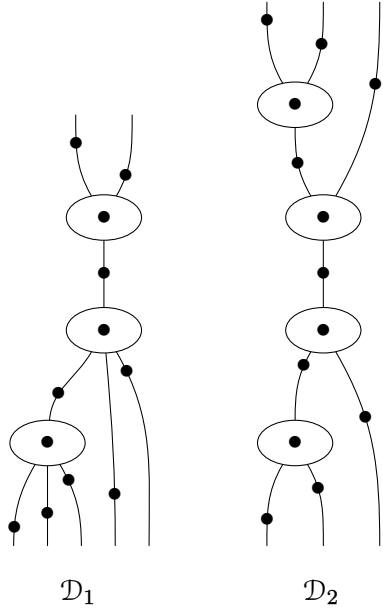
\begin{figure}[h!]
    \centering
    \begin{tabular}{ccc}
    \begin{tikzpicture}[line cap=round,line join=round,x=0.75cm,y=0.75cm]
        % sommets
        \floor (F1) at (-1,0);
        \floor (F2) at (0,2);
        \floor (F3) at (0,4);
        \node (T1) at (-0.5,6) {};
        \node (T2) at (0.5,6) {};
        \node (B1) at (-1.6,-2) {};
        \node (B2) at (-1,-2) {};
        \node (B3) at (-0.4,-2) {};
        \node (B4) at (0.2,-2) {};
        \node (B5) at (0.8,-2) {};
        % arêtes
        \marked (T1) to (F3) pos=0.33 in=120 out=-90 ;
        \marked (T2) to (F3) pos=0.75 in=60 out=-90 ;

        \marked (F3) to (F2) pos=0.5 in=90 out=-90 ;
        \marked (F2) to (F1) pos=0.66 in=85 out=-120;
        
        \marked (B1) to (F1) pos=0.2 in=-120 out=90 ;
        \marked (B2) to (F1) pos=0.4 in=-90 out=90 ;
        \marked (B3) to (F1) pos=0.8 in=-60 out=90 ;
        \marked (B4) to (F2) pos=0.25 in=-85 out=90 ;
        \marked (B5) to (F2) pos=0.9 in=-60 out=90 ;
    \end{tikzpicture} & & \begin{tikzpicture}[line cap=round,line join=round,x=0.75cm,y=0.75cm]
        % sommets
        \floor (F1) at (-0.5,0);
        \floor (F2) at (0,2);
        \floor (F3) at (0,4);
        \floor (F4) at (-0.5,6);
        \node (T1) at (-1,8) {};
        \node (T2) at (0,8) {};
        \node (T3) at (1,8) {};
        \node (B1) at (-1,-2) {};
        \node (B2) at (0,-2) {};
        \node (B3) at (1,-2) {};
        % arêtes
        \marked (T1) to (F4) pos=0.2 in=120 out=-90 ;
        \marked (T2) to (F4) pos=0.5 in=60 out=-90 ;
        \marked (T3) to (F3) pos=0.4 in=60 out=-90 ;

        \marked (F4) to (F3) pos=0.5 in=120 out=-90 ;
        \marked (F3) to (F2) pos=0.5 in=90 out=-90 ;
        \marked (F2) to (F1) pos=0.2 in=90 out=-120 ;
        
        \marked (B1) to (F1) pos=0.3 in=-120 out=90 ;
        \marked (B2) to (F1) pos=0.7 in=-60 out=90 ;
        \marked (B3) to (F2) pos=0.65 in=-60 out=90 ;
    \end{tikzpicture} \\
    $\D_1$ & $\ \ \ \ $ & $\D_2$ \\
    \end{tabular}

    \caption{The two marked diagrams corresponding to the words from Example \ref{expl-correspondence-word-diagram-genus-0-hirzebruch}.}
    \label{fig:correspondence-words-marked-diagrams}
\end{figure}

\begin{expl}\label{expl-correspondence-word-diagram-genus-0-hirzebruch} 
On Figure \ref{fig:correspondence-words-marked-diagrams}, we have two genus $0$ marked diagrams with totally ordered floors corresponding to the words
\begin{align*}
    W(\D_1) &=  \sfb_0\sfb_0\sfb_1\sfb_0\sff\sfe\sfb_1\sff\sfe\sff\sft_0\sft_0, \\
    W(\D_2) &=  \sfb_0\sfb_0 \sff\sfb_1\sfe\sff\sfe\sff\sfe\sff\sft_1\sft_0\sft_0. 
\end{align*}
Note that it is possible to recover the marked diagrams from the words.
\end{expl}

This correspondence between diagrams and words is in fact bijective provided we have some assumptions on the words.

\begin{prop}\label{prop-corresp-word-diag-genus-0-hirzebruch}
    Let $\D$ be a marked floor diagram in the class $\beta=aE+bF\in H_2(\FF_\delta,\ZZ)$. Then the word $W(\D)$ satisfies the following.
        \begin{enumerate}[label=(\roman*)]
            \item Forgetting about the letters $\sfb_*$ and $\sft_*$, the word is just
            $$(\sff\sfe)^{a-1}f = \sff\sfe\sff\sfe\cdots\sff\sfe\sff. $$
            Moreover, there are $b$ letters $\sfb_*$ and $b+\delta a$ letters $\sft_*$.
            \item Given a letter $\sfb_k$, assume the word forgetting the $\sfe$, $\sft_j$ and remaining $\sfb_j$ is $\sff^p \sfb_k \sff^{a-p}$, then we have $k\geqslant p$.
            \item Given a letter $\sft_k$, assume the word forgetting the $e$, $\sfb_j$ and remaining $\sft_j$ is $\sff^{a-p} \sft_k \sff^p$, then we have $k\geqslant p$.
        \end{enumerate}
    Conversely, a word satisfying the above conditions yields a marked floor diagram for which the floors are totally ordered. The set of words satisfying the above conditions is denoted by $\WWW_{aE+bF}$.
\end{prop}

\begin{proof}
\begin{enumerate}[label=(\roman*)]
    \item The diagram has $a$ floors and they are totally ordered, so that each floor is linked to the next one by a unique elevator. Thus forgetting about $\sft_*$ and $\sfb_*$, we get $\sff\sfe\sff\cdots\sff\sfe\sff$. The numbers of floors as well as the number of ends in each direction are fixed by the class $aE+bF$.
    \item In the word $\sff^p \sfb_k \sff^{a-p}$, the marking of the end encoded by $\sfb_k$ lies between the floors $p$ and $p+1$. Thus, the end being a bottom end, it skips at least the $p$ first floors and is attached to a floor after the $(p+1)$-th floor, so that $k\geqslant p$.
    \item  The reasoning is the same but with top ends instead of bottom ends.
\end{enumerate}

\medskip

For the converse construction, let $\sfW$ be a word satisfying (i)-(iii). We start with the ordered graph having $a$ vertices, each linked to the next one by a unique edge, and with a marking. For each $\sfb_j$ (resp. $\sft_j$) we insert a bottom (resp. top) end attached to the floor $j+1$ (resp. $a-j$) with a marking lying at the corresponding place in the word. There is a unique way to add weights to the bounded edges so that the diagram is balanced. Condition (i) ensures that the diagram has the right number of floors and ends, and the conditions (ii) and (iii) ensures that it is possible to place the marking of an end on the latter.
\end{proof}

\subsubsection{Words and codegrees.} We define the codegree function on $\WWW_{aE+bF}$ so that the codegree of a word matches the codegree of the associated diagram. Let $\WWW$ be the set of all words on the considered alphabet, which is a monoid. The codegree function is actually the restriction of the following morphism of monoids:
\[  \begin{array}{cccc} \codeg : &   \WWW & \longrightarrow & \NN \\
      & \sft_j,\ \sfb_j & \longmapsto & j \\
     & \sfe,\ \sff & \longmapsto & 0 
\end{array} \]
and by construction we have $\codeg(\D)=\codeg(W(\D))$.

\begin{rem}
The definition of $\WWW_{aE+bF}$ allows the letters $\sfb_*$ and $\sft_*$ to interlace, meaning there might be a $\sfb_*$ after a $\sft_*$. However, if a $\sfb_*$ lies after a $\sft_*$ then all floors are skipped by at least one of these two ends so that $\codeg(W(\D))\geqslant a$. If we restrict to words of codegree at most $i$ and if $a>i$, then this situation does not appear.
\end{rem}

The following lemma describes the shape of the words that have a bounded codegree provided the class is large enough.

\begin{lem}\label{lem-BT-words-if-bounded-codegree}
Assume $i\geqslant 1$ and $a>2i$. The words in $\WWW_{aE+bF}$ of codegree at most $i$ are of the following form:
    \begin{align*}
         & \ \sfB_0\left[\prod_{j=1}^i \sff\sfB^{(1)}_j\sfe\sfB^{(2)}_j\right](\sff\sfe)^{a-2i}\left[\prod_{j=1}^i \sff\sfT^{(1)}_{i+1-j}\sfe\sfT^{(2)}_{i+1-j}\right]\sff\sfT_0 \\
        =& \ \sfB_0\sff\sfB_1^{(1)}\sfe\sfB_1^{(2)}\sff\sfB_2^{(1)}\sfe\sfB_2^{(2)}\cdots \sff\sfe\sff\sfe\sff\sfe\cdots \sfT_2^{(1)}\sfe\sfT_2^{(2)}\sff\sfT_1^{(1)}\sfe\sfT_1^{(2)}\sff\sfT_0,
    \end{align*}
where $\sfB_0, \sfB_j^{(k)}$ (resp. $\sfT_0, \sfT_j^{(k)}$) are words in the letters $\{\sfb_*\}_{* \geq 0}$ and $\{\sfb_*\}_{*\geq j}$ (resp. $\{\sft_*\}_{*\geq 0}$ and $\{\sft_*\}_{*\geq j}$).
\end{lem}

\begin{proof}
As a letter $\sfb_*$ put after the first $i+1$ letters $\sff$ contributes at least $i+1$ to the codegree, it cannot appear if the latter is assumed to be smaller than $i$, and similarly for $\sft_*$ letters.
\end{proof}

Basically, the word has a core $(\sff\sfe)^{a-1}f$ and we insert a word in the letters $\sfb_*$ (called $B$-word) between each of the $2i$ consecutive letters on the left, a word in the letters $\sft_*$ (called $T$-word) similarly on the right. As the roles of $B$-words and $T$-words are symmetric, we call them ``end-words''. We denote by $\S$ the set of sentences, \ie of families of end-words in $\sfs_*$ where $\sfS,\sfs$ are meant to be replaced by $\sfT,\sft$ or $\sfB,\sfb$:
$$\S=\{ (\sfS_0,\sfS_1^{(1)},\sfS_1^{(2)},\dots,\sfS_i^{(1)},\sfS_i^{(2)})\ |\ i\geqslant 0,\ \sfS_j^{(k)} \text{ word in }\{\sfs_*\}_{*\geqslant j} \}.$$
It is endowed with functions
\begin{align*}
    \codeg : &\ \S \longrightarrow \NN , \\
    \ell_0,\ \ell_j^{(k)} : &\ \S \longrightarrow \NN , \\
    \ell : &\ \S \longrightarrow \NN , 
\end{align*}
that associate to a sentence in $\S$ the sum of the codegrees of its words, the length of the words $\sfS_0$ and $\sfS_j^{(k)}$ (maybe $0$), and the sum of their lengths. For $n\geq0$ we denote by $\S(n)$ the set of sentences with total length $n$.

Lemma \ref{lem-BT-words-if-bounded-codegree} asserts that choosing a word in $\WWW_{aE+bF}$ having codegree at most $i$ and with $a,b$ large enough amounts to choose :
\begin{itemize}
    \item an element $\bfk \in \S(b)$ that encodes the $B$-words,
    \item an element $\tfk \in \S(b+\delta a)$ that encodes the $T$-words,
\end{itemize}
such that $\codeg(\tfk)+\codeg(\bfk)\leq i$.
Essentially, elements of $\S(b)$ and $\S(b+\delta a)$ tell us how to construct half-diagrams which are glued back together. Hence, the computation of a generating series over $\WWW_{aE+bF}$ will split into the computations of some generating series over $\S(b)$ and $\S(b+\delta a)$.

\begin{defi}
We define the \emph{multiplicity} of a sentence $\sfk \in \S(n)$ to be
\[ \mu_{\S(n)}(\sfk)=(1-x)^n x^{\codeg(\sfk)} . \]
\end{defi}

\subsubsection{Enumeration of words.} We now compute the generating series of sentences with their multiplicity.

\begin{lem}\label{lem-multiplicity-asymptotic-word-genus-0-hirzebruch}
Assume $i\geqslant 1$ and $a,b>2i$. The multiplicity modulo $x^{i+1}$ of the diagram $\D$ encoded by a word $\sfW\in\WWW_{aE+bF}$ is $(1-x)^{2b+\delta a}x^{\codeg(\sfW)}$.
\end{lem}

\begin{proof}
By definition, the multiplicity is
$$(1-x)^{2b+\delta a}x^{\codeg(\sfW)}\prod_e (1-x^{w(e)})^2,$$
where the product is over the bounded edges $e$ of $\D$. Assume $\codeg(\sfW)\leqslant i$, otherwise there is nothing to prove since we get $0$ modulo $x^{i+1}$. By Lemma \ref{lem-poids-grands}, the unique edge between two consecutive floors has weight bigger than $i$. Thus, $(1-x^{w(e)})^2\equiv 1\modulo x^{i+1}$.
\end{proof}

\begin{lem}\label{lem-generating-series-T-words}
Let $n>i\geqslant 1$. The generating series of length $n$ sentences counted with their multiplicity is
\[ (1-x)^n \sum_{\sfk\in\S(n)} x^{\codeg(\sfk)} \equiv p(x)^2 \modulo x^{i+1}.\]
\end{lem}

\begin{proof}
As we are looking at an equality modulo $x^{i+1}$, we only care about the elements of $\S(n)$ with codegree at most $i$ since the others elements contribute $0$. In particular, each sentence contains at most $2i+1$ words and the letters involved in each word can only be in $\{\sfs_*\}_{0\leqslant\ast\leqslant i+1}$, so that the sum on the left is well-defined modulo $x^{i+1}$.

Let's fix $(l_0,l_j^{(k)})_{\substack{1\leqslant j\leqslant n \\ k=1,2}}$ a family of integers such that $n=l_0+\sum_{j,k} l_j^{(k)}$, and look at sentences $\sfk = (\sfS_0,\sfS_1^{(1)},\dots,\sfS_n^{(2)})$ with $\ell_0(\sfk) = l_0$ and $\ell_j^{(k)}(\sfk) = l_j^{(k)}$. The sum of the multiplicities of such sentences is
\[ (1-x)^n \left( \dsum_{\ell(\sfS_0)=l_0} x^{\codeg(\sfS_0)} \right) \times \dprod_{j,k} \left( \dsum_{\ell(\sfS_j^{(k)}) = l_j^{(k)}} x^{\codeg(\sfS_j^{(k)})} \right). \]
Letters in $\sfS_0$ (resp. $\sfS_j^{(k)}$) can take values in $\{\sfs_*\}_{*\geqslant 0}$ (resp. $\{\sfs_*\}_{*\geqslant j}$), so one has 
\[ \begin{array}{rl}
     & \dsum_{\ell(\sfS_0)=l_0} x^{\codeg(\sfS_0)} = \left(\dsum_{k\geq0} x^k \right)^{l_0} = \left(\dfrac{1}{1-x}\right)^{l_0} \\
    \text{and} & \dsum_{\ell(\sfS_j^{(k)})=l_j^{(k)}} x^{\codeg(\sfS_j^{(k)})} = \left(\dsum_{k\geq j} x^k \right)^{l_j^{(k)}} = \left(\dfrac{x^j}{1-x}\right)^{l_j^{(k)}} 
\end{array} \]
and the sum of the multiplicities of the sentences with fixed lengths equal to $(l_0,l_j^{(k)})_{j,k}$ is
\[ (1-x)^n \left(\dfrac{1}{1-x}\right)^{l_0} \dprod_{j,k} \left(\dfrac{x^j}{1-x}\right)^{l_j^{(k)}} = \prod_{j,k} x^{jl_j^{(k)}}. \]

It remains to sum over all the possible choices of $(l_0,l_j^{(k)})_{j,k}$. Because the total length of the sentences is $n$, we can forget about $l_0$ since it is fully determined by the $l_j^{(k)}$. Moreover we can sum over $l_j^{(k)} \geq 0$ instead of $\sum l_j^{(k)} = n$ because the excess terms will contribute 0 modulo $x^n$. Therefore, we get
\[ \sum_{(l_0,l_j^{(k)})} x^{\sum jl_j^{(k)}} = \prod_{j,k} \sum_{l_j^{(k)}\geq0} x^{jl_j^{(k)}} =  \left(\prod_{j=1}^n \frac{1}{1-x^j}\right)^2 \\
\equiv p(x)^2 \modulo x^n.
\]
\end{proof}

\begin{theo}\label{theo-AR-genus0-hirzebruch}
The genus $0$ asymptotic refined invariant of the Hirzebruch surface $\FF_\delta$ is
\[ AR_{0}^{\FF_\delta}= p(x)^4, \]
where $p(x)$ is the generating series of partition numbers.
\end{theo}

\begin{proof}
We can determine $AR_{0}^{\FF_\delta} \mod x^{i+1}$ by summing the multiplicities of the words of $\WWW_{aE+bF}$ of codegree at most $i$, and with $a,b>2i$. According to Lemma \ref{lem-BT-words-if-bounded-codegree}, choosing a word of codegree at most $i$ amounts to choose sentences $\bfk\in\S(b)$ and $\tfk\in\S(b+\delta a)$ with $\codeg(\bfk) + \codeg(\tfk) \leq i$. 
Lemma \ref{lem-multiplicity-asymptotic-word-genus-0-hirzebruch} ensures that the multiplicity of the word is
\[ (1-x)^{b} x^{\codeg(\bfk)}(1-x)^{b+\delta a} x^{\codeg(\tfk)}. \]
Hence, summing over $\S(b) \times \S(b+\delta a)$ (and potentially counting terms which contribute 0 modulo $x^{i+1}$) the generating series factors modulo $x^{i+1}$ :
\[  \left((1-x)^{b} \sum_{\bfk\in\S(b)} x^{\codeg(\bfk)}\right)
    \left((1-x)^{b+\delta a}\sum_{\tfk\in\S(b+\delta a)} x^{\codeg(\tfk)}\right) .\]
Using Lemma \ref{lem-generating-series-T-words} we get the result modulo $x^{i+1}$ for any $i$, and we conclude.
\end{proof}

    \subsection{The case of $h$-transverse toric surfaces}
\label{sec-h-transverse-genus-0}
We now consider the case of a toric surface $X$ associated to a $h$-transverse, horizontal and non-singular polygon $\Delta$. Let $\beta\in H_2(X,\ZZ)$ be the corresponding homology class. Given a primitive vector $\alpha$ that positively generates a ray of the dual fan of $\Delta$, we denote by $D_\alpha$ the corresponding toric divisor.

\subsubsection{Words and codegree for $h$-transverse polygons.}
The main difference with the Hirzebruch case, is that marked floor diagrams are modified by incorporating the data $(L,R)$, \ie assigning a pair of integers called \textit{sloping pair} to each floor. According to \cite[Section 3]{brugalle2020polynomiality}, the codegree coming from the sloping pairs is
\[ \codeg(L,R)=\sum_{\substack{ v\prec v' \\ \text{s.t. } L(v)>L(v') } }(L(v)-L(v')) + \sum_{\substack{v \prec v' \\ \text{s.t. } R(v)>R(v') }}(R(v)-R(v')). \]
Elements in each of the sums are called \textit{inversions}. In particular, the contribution to the codegree is $0$ if $L$ and $R$ are increasing.

\medskip

To enable the word approach to treat the case of $h$-transverse polygons, we need to add a sloping pair to each floor. We now consider the alphabet $\{\sfe,\sff_{*,*},\sft_*,\sfb_*\}$ where the indices of $\sff_{*,*}$ are the members of the sloping pair. Similarly to Proposition \ref{prop-corresp-word-diag-genus-0-hirzebruch}, we have the following lemma that relates words to diagrams.

\begin{prop}\label{prop-corresp-word-diag-genus-0-general}
Let $\D$ be a marked floor diagram in the class $\beta\in H_2(X,\ZZ)$. Then the word $W(\D)$ satisfies the following.
    \begin{enumerate}[label=(\roman*)]
        \item Forgetting about $\sfb_*$, $\sft_*$ and indices of $\sff_{\ast,\ast}$, the word is $(\sff\sfe)^{a-1}\sff$. Moreover, there are $b^\top$ letters $\sft_*$ and $b^\bot$ letters $\sfb_*$.
        \item[(ii)-(iii)] from Proposition \ref{prop-corresp-word-diag-genus-0-hirzebruch} are still satisfied.
        \item[(iv)] If $k\in b_{\text{left}}(\Delta_\beta)$ (resp. $b_{\text{right}}(\Delta_\beta)$), the number of appearances of $k$ as a $L$-value (resp. $R$-value) in the sloping pairs is $\beta \cdot D_\alpha$, where $\alpha=(-1,k)$ (resp. $(1,k)$).
    \end{enumerate}
We denote by $\WWW_\beta$ the set of words satisfying the above conditions. Given a word $\sfW\in\WWW_\beta$, there is a unique way to recover a marked floor diagram in the class $\beta$ potentially with negative weights.
\end{prop}

\begin{proof}
The proof of the first three points is verbatim to those of Proposition \ref{prop-corresp-word-diag-genus-0-hirzebruch}. The last one results from the definition of sloping pairs. For the converse construction, we also proceed as in Proposition \ref{prop-corresp-word-diag-genus-0-hirzebruch}. The difference is that when adding the weights of the elevators, we may obtain negative or zero weights. 
\end{proof}

\begin{rem}
    During the reconstruction, the weights that appear may be negative. However, for the words of $\WWW_\beta$ that we will consider all the weights are positive, see Lemma \ref{lem-sloping-pairs-if-bounded-codegree}.
\end{rem}

In a word $\sfW$, we say that two letters $\sff_{\ell,r}$ and $\sff_{\ell',r'}$ appearing in that order form a \textit{left inversion} (resp. \textit{right inversion}) if $\ell>\ell'$ (resp. $r>r'$). The size of this inversion is the quantity $\ell-\ell'$ (resp. $r-r'$).

The codegree function on $\WWW_\beta$ is defined to match the codegree of the marked floor diagrams. The difference with the Hirzebruch case is that the codegree comes from the $T$-words and $B$-words,
but also from the sloping pairs:
$$ \begin{array}{crcl}\codeg :&   \WWW_\beta & \longrightarrow & \NN \\
    & \sfW & \longmapsto & \codeg(\mathrm{ft}(\sfW))+\codeg(L,R)
\end{array} $$
where $\mathrm{ft}(\sfW)$ is the word where we forget the indices of the letters $\sff_{\ast,\ast}$.

\medskip

Up to the indices of letters $\sff_{*,*}$, Lemma \ref{lem-BT-words-if-bounded-codegree} still applies for words in $\WWW_\beta$ under the hypothesis $a>2i$. 
We deal with the indices of letters $\sff_{*,*}$ in the following lemma.

\begin{lem}\label{lem-sloping-pairs-if-bounded-codegree}
Let $i\geqslant 1$ and assume that for each toric divisor $D$ we have $\beta\cdot D>2i$. If $\sfW\in\WWW_\beta$ has codegree at most $i$ then:
    \begin{enumerate}[label=(\roman*)]
        \item all the inversions in the sloping pairs are of size one, \ie correspond to consecutive sides of the polygon,
        \item two letters $\sff_{\ast,\ast}$ part of an inversion are separated by at most $i-1$ letters $\sff_{*,*}$,
        \item the weights of the elevators in the associated diagram are strictly bigger than $i$. In particular they are positive, so that $\sfW$ corresponds to a true marked diagram.
    \end{enumerate}
\end{lem}

\begin{proof}
We denote by $(L,R)$ the tuple of sloping pairs of $\sfW$. We first notice that the tuple $L$ (and similarly for $R$) differs from the unique tuple of increasing slopes by a finite number of transpositions that switches two consecutive elements. Indeed, this is true for the tuple of codegree 0 since in that case this tuple is increasing.

If we consider a tuple of positive codegree then there is a consecutive pair that forms an inversion; if not, the tuple would be increasing. Then, switching both members of the inversion decreases the codegree, and we conclude by induction.

Each transposition switching consecutive elements increases the codegree by at least $1$, so that if $\codeg(\sfW) \leq i$ then $L$ differs from the increasing tuple by at most $i$ transpositions.

As we assume the lengths of the sides of $\Delta_\beta$ to be bigger than $2i$, it is not possible to create an inversion of size bigger than $2$ with only $i$ transpositions, proving $(i)$.

Take an inversion $(\dots,k+1,\dots,k,\dots)$ with $i$ elements in-between. Any of these $i$ elements is either $k$ or $k+1$. If it is a $k$ it provides an inversion with the left $k+1$, and if it is a $k+1$ it provides an inversion with the right $k$. Hence we get at least $1+i$ inversion, which is impossible, proving $(ii)$.

Finally, for $(iii)$ if the codegree is $0$ then the assumption ensures that the weights of all the elevators are bigger than $2i$ by Lemma \ref{lem-poids-grands}. If not, each transposition decreases the weight of an edge by $1$. Thus, they remain strictly bigger than $i$ after $i$ transpositions.
\end{proof}

\subsubsection{Encoding the sloping pairs.}\label{sec-definition-pearl}

Proposition \ref{prop-corresp-word-diag-genus-0-general} states that the elements of the sloping pairs are assigned to the floors with some constraints. We use the following objects to encode these assignments. Let $\P$ be the set of non-constant sequences $\pfk\in\{\bullet,\circ\}^\ZZ$ up to reindexation by translation of the index such that the set of pairs
$$I(\pfk)=\{(k,l)\ |\ k<l,\pfk_k=\circ,\pfk_l=\bullet\},$$
is finite. These pairs are also called \textit{inversions}. We then set $\codeg(\pfk)=|I(\pfk)|$.

\begin{expl}\label{expl-pearl-chain}
We consider the following element, for which the first $\circ$ has index $0$:
$$\pfk=\cdots\bullet\bullet\circ\circ\bullet\circ\bullet\bullet\circ\bullet\circ\circ\circ\cdots.$$
Since 
\[ I(\pfk) =  \{ (0,7),(1,7),(3,7), (6,7), (0,5),(1,5),(3,5), (0,4),(1,4),(3,4),(0,2),(1,2) \} \]
      it has codegree 12.
\end{expl}

We notice that for each element $\pfk\in\P$, as it is non-constant it contains at least a $\circ$ and a $\bullet$. Since there is a finite number of pairs $\circ\prec\bullet$ (i.e. a pair $k<l$ with $\pfk_k=\circ$ and $\pfk_l=\bullet$), the sequences is asymptotically constant to $\circ$ near $+\infty$, and $\bullet$ near $-\infty$.

\begin{lem}\label{lem-generating-series-corners}
Let $i\geqslant 1$. There is a finite number of elements of $\P$ with codegree smaller than $i$, and one has
    \[ \sum_{\pfk\in\P}x^{\codeg(\pfk)}=p(x).\]
\end{lem}

\begin{proof}
Let $\pfk\in\P$ be a sequence with codegree smaller than $i$. Choose the reindexation of $\pfk$ such that $i$ is the last index whose value is $\bullet$. As $I(\pfk)$ is finite, there is a finite number of $\circ$ before index $i$ since each of them yields an inversion. Moreover, none can have negative index otherwise we would have the form
$$\pfk=\cdots\bullet\bullet\bullet \cdots\circ [\cdots]\bullet\circ\circ\circ\cdots,$$
and each element in the bracketed zone yields an inversion,
leading to more than $i+1$ inversions. Thus the set $\{\pfk\in\P\ |\ \codeg(\pfk)\leqslant i\}$ is finite and the generating series is well-defined.

An element $\pfk \in\P$ is fully determined by the sequence with finite support $u(\pfk)=(u_j)_{j\geqslant 1}$, with $u_j$ being the number of $\bullet$ with $j$ $\circ$ on their left. The inverse bijection associates to an integer sequence with finite support $u$ the element of $\P$ defined as follows:
        \begin{itemize}[label=$\ast$]
            \item put a $\circ$ at $0$ and $\bullet$ for negative indices,
            \item inductively, starting at $j=1$, put $u_j$ $\bullet$ and then a new $\circ$,
            \item as $u$ is of finite support, the algorithm finishes by only putting $\circ$.
        \end{itemize}
The codegree expresses as
\[ \codeg(\pfk)=\sum_{j=1}^\infty ju_j.\]
If $\codeg(\pfk)\leq i$ then $u_j=0$ for $j>i$. Computing the generating series modulo $x^{i+1}$, we only care about the $\pfk$ having the sequence $u(\pfk)$ with support in $[\![1;i]\!]$, and $u_j$ may take any value considered that too large values  will contribute $0$ modulo $x^{i+1}$. Thus one has:
\begin{align*}
\sum_{\pfk\in\P}x^{\codeg(\pfk)}\equiv &  \sum_{u_1,\dots,u_i=0}^\infty x^{\sum ju_j} \modulo x^{i+1} \\
\equiv & \prod_{j=1}^i \left( \sum_{u_j=1}^\infty x^{ju_j}\right)=\prod_{j=1}^i \frac{1}{1-x^j} \modulo x^{i+1} \\
\equiv & \prod_{j=1}^\infty\frac{1}{1-x^j} = p(x) \modulo x^{i+1}.
\end{align*}
As the congruence is true modulo $x^{i+1}$ for every $i$, we get the desired equality.
\end{proof}

\begin{expl}
Continuing Example \ref{expl-pearl-chain} one has $u(\pfk)=(0,1,2,1,0,0,\dots).$
\end{expl}

\begin{lem}\label{lem-correspondence-LR-inv}
    Let $\sfW\in\WWW_\beta$ with $\codeg(\sfW)\leqslant i$ and $\beta\cdot D>2i$ for any toric divisor $D$.  Then the data of the sloping pairs $(L,R)$ is equivalent to the data of an element $\pfk_c \in\P$ for any corner of $\Delta_\beta$ non-adjacent to a horizontal edge, such that $\codeg(L,R) = \sum_c \codeg(\pfk_c)$.
\end{lem}

\begin{proof}
    Let $(L,R)$ be the tuple of sloping pairs of $\sfW$, and let $\theta\leqslant p\leqslant\theta'$ be the integers such that the edges of the left side of $\Delta_\beta$ have outgoing normal vectors $(-1,p)$. Point $(i)$ of lemma \ref{lem-sloping-pairs-if-bounded-codegree} says that $L$ writes as a concatenation $L = (L_\theta,\dots,L_{\theta'-1})$ where $L_p$ is of the form $(p,\dots,p, \star,\dots,\star,p+1,\dots,p+1)$ with $\star \in\{p,p+1\}$. Given $p$, let $c_p^-$ be the corner of $\Delta_\beta$ whose adjacent edges have outgoing normal vectors $(-1,p)$ and $(-1,p+1)$. Replacing $p$ by $\bullet$ and $p+1$ by $\circ$, the tuple $L_p$ gives an element $\pfk_{c_p^-} \in \P$. Similarly, $R$ gives elements $\pfk_{c_p^+} \in \P$. By construction, one has $\codeg(L,R) = \sum_c \codeg(\pfk_c)$, where the sum runs over the corners of $\Delta_\beta$ non-adjacent to a horizontal edge.

    Conversely, assume we are given a family $(\pfk_c)_c \in \P^{\chi-4}$. We construct $L$ from the elements $\pfk_{c_p^-}$ corresponding to corners of the left side of $\Delta_\beta$ in the following way. For any $p$, truncate $\pfk_{c_p^-}$ just before its first $\circ$ and just after its last $\bullet$. Replacing $\bullet$ by $p$ and $\circ$ by $p+1$ gives a tuple $\tilde L_p$. Then $L$ is the concatenation $L = (\tilde L_\theta, \dots, p,p, \tilde L_p, p+1, p+1,\dots,L_{\theta'-1})$ where we add sufficiently enough $p$ between $\tilde L_{p-1}$ and $\tilde L_p$ so that the total number of $p$ is the number given by proposition \ref{prop-corresp-word-diag-genus-0-general} $(iv)$. We proceed similarly for $R$, and by construction one has $\codeg(L,R) = \sum_c \codeg(\pfk_c)$.
\end{proof}

\begin{expl}\label{expl-pearls-from-sloping-pairs}
   To the tuple $L = (0,1,0,1,1,0,1,1,1,2,1,2,2)$ we associate the sequences $\pfk_1 = \cdots\bullet\bullet\circ\bullet\circ\circ\bullet\circ\circ\cdots$ and $\pfk_2= \cdots\bullet\bullet\circ\bullet\circ\circ\cdots$,  where $\bullet$ and $\circ$ correspond to 0 and 1 in $\pfk_1$ (resp. 1 and 2 in $\pfk_2$).
\end{expl}

\subsubsection{Enumeration of words in the $h$-transverse setting.}

We can now compute the asymptotic refined invariant in genus $0$ for $h$-transverse polygons.

\begin{theo}\label{theo-AR-genus0-general}
    Let $X$ be a toric surface associated to a $h$-transverse, horizontal and non-singular polygon, with Euler characteristic $\chi$. Then the genus 0 asymptotic refined invariant is
    $$AR^X_{0}= p(x)^\chi.$$
\end{theo}

\begin{proof}
We can determine $AR_0^X \mod x^{i+1}$ by summing the multiplicities of the words of $\WWW_\beta$ of codegree at most $i$, with $\beta\in H_2(X,\ZZ)$ such that for every toric divisor $D$ we have $\beta\cdot D>2i$.

By Lemma \ref{lem-sloping-pairs-if-bounded-codegree} the weight of every bounded elevator in the diagram associated to a word $\sfW\in\WWW_\beta$ of codegree at most $i$ is strictly bigger than $i$. Hence the multiplicity modulo $x^{i+1}$ is
$$(1-x)^{b^\top+b^\bot} x^{\codeg(\sfW)}.$$
The word is fully determined by the following data:
    \begin{itemize}
        \item an element $\tfk\in\S(b^\top)$ encoding the $T$-words,
        \item an element $\bfk\in\S(b^\bot)$ encoding the $B$-words,
        \item an element $\pfk_c\in\P$ for any of the $\chi-4$ corners $c$ of $\Delta$ non-adjacent to a horizontal side,
    \end{itemize}
such that
\[ \codeg(\sfW) = \codeg(\tfk) + \codeg(\bfk) + \dsum_c \codeg(\pfk_c) \leq i.  \]
The data of $\tfk$ and $\bfk$ are enough to recover the word up to the indices of the letters $\sff_{*,*}$. The data of the $\pfk_c$ allows to recover the sloping pairs $(L,R)$ by Lemma \ref{lem-correspondence-LR-inv}. Hence, summing over $\S(b^\bot) \times \S(b^\top) \times \P^{\chi-4}$ (and potentially counting terms which contribute 0 modulo $x^{i+1}$) the generating series of words counted with multiplicity factors modulo $x^{i+1}$ :
\begin{align*}
  \left((1-x)^{b^\bot}\sum_{\bfk\in\S(b^\bot) } x^{\codeg(\bfk)}\right) 
 \left((1-x)^{b^\top}\sum_{ \tfk\in\S(b^\top) } x^{\codeg(\tfk)}\right)    \left( \sum_{\pfk\in\P }x^{\codeg(\pfk)}\right)^{\chi-4}.
 \end{align*}
Using Lemmas \ref{lem-generating-series-T-words} and \ref{lem-generating-series-corners} we obtain for the generating series
\[  p(x)^2\cdot p(x)^2\cdot p(x)^{\chi-4} =  p(x)^\chi \modulo x^{i+1}. \]
As this is true for every $i\geqslant 1$ we get the result.
\end{proof}

\section{Asymptotic refined invariant in genus $1$}
\label{sec-genus1}

The idea to compute the genus $1$ asymptotic invariant is to construct floor diagrams of genus $1$ by adding an edge to a genus $0$ diagram. This way, we can group together the genus $1$ diagrams obtained from the same genus $0$ diagram, so that we reduce the enumeration to the genus $0$ case, with a multiplicity corresponding to the weighted count of diagrams. We start with Hirzebruch surfaces before going to $h$-transverse, horizontal and non-singular toric surfaces. The strategy is the same: we compute $\widetilde{BG}_1^X(\beta)$ modulo $x^{i+1}$ and find an expression that does not depend on $i$, before making $i$ goes to $\infty$.

    \subsection{The case of Hirzebruch surfaces}

    To get to the genus $1$ case, the idea is that a genus $1$ diagram is obtained from a genus $0$ diagram by adding one edge, and conversely we get a genus $0$ diagram by removing an edge from a genus $1$ diagram. However, it might not be clear which edge to remove, and what to do to balance the diagram again. We make this construction precise by introducing the notion of \textit{nerved diagram}.

\subsubsection{Nerved diagrams.} We already fixed an integer $i$ to bound the codegree of diagrams we look at. Let us fix a second integer $M\geq1$.

        \begin{defi}
            Let $\D$ be a genus $g$ diagram in a class $aE+bF$, with $a>2i$ and $b>(g+1)M+i$. Assume $\codeg(\D) \leq i$. A \textit{nerve} for $\D$ is the choice of an edge between each pair of consecutive floors with weight $\geqslant M$. We call the data of $\D$ with the choice of a nerve a \textit{nerved diagram}. We denote with a tilde the nerved diagrams, \eg $\tilde\D$.
        \end{defi}

\begin{rem}
    For genus $g$, provided $b>(g+1)M+i$ and $\codeg(\D) \leq i$, Lemma \ref{lem-poids-grands} ensures the existence of a nerve.
\end{rem}

\begin{lem}
Assume $b>i+2M$ and let $\D$ be a floor diagram in the class $aE+bF$ with $\codeg(\D) \leq i$. 
\begin{enumerate}[label=(\roman*)]
    \item If $\D$ is of genus $0$, there exists a unique choice of nerve.
    \item If $\D$ is of genus $1$ with an edge skipping some floors, there exists a unique choice of nerve.
    \item If $\D$ is of genus $1$ with two edges linking consecutive floors, there are one or two possible nerves depending on whether only one of the edges or both have weight bigger than $M$.
\end{enumerate}
\end{lem}

\begin{proof}
\begin{enumerate}[label=$(\roman*)$]
    \item In the genus $0$ case, we already know by \cite{brugalle2020polynomiality} that the floors are totally ordered in the diagram. The total weight between the floors $m$ and $m+1$ is $b+\delta m$ minus the number of sinks that skip the floor $m+1$ and the number of sources that skip the floor $m$. As the number of ends skipping some floors is bounded by $i$, the weight of the unique edge between two consecutive floors is bigger than $b-i\geqslant M$, so that there is a unique nerve.
    \item  Because $\codeg(\D) \leq i$, the weight of the skipping edge is bounded by $i$ and we conclude as in the genus $0$ case.
    \item  The sum of the weights of the two edges is bigger than $b-i>2M$, so that at least one of them has weight $\geqslant M$.
    \end{enumerate}
\end{proof}

\begin{figure}[h!]
    \centering
    \begin{tabular}{ccccc}
    \begin{tikzpicture}[line cap=round,line join=round,x=0.75cm,y=0.75cm]
        % sommets
        \floor (F1) at (-1,0);
        \floor (F2) at (0,2);
        \floor (F3) at (0,4);
        \node (T1) at (-0.5,6) {};
        \node (T2) at (0.5,6) {};
        \node (B1) at (-1.6,-2) {};
        \node (B2) at (-1,-2) {};
        \node (B3) at (-0.4,-2) {};
        \node (B4) at (0.2,-2) {};
        \node (B5) at (0.8,-2) {};
        % arêtes
        \marked (T1) to (F3) pos=0.33 in=120 out=-90 ;
        \marked (T2) to (F3) pos=0.75 in=60 out=-90 ;

        \leftmarked (F3) to (F2) pos=0.5 in=90 out=-90 w=3;
        \leftmarked (F2) to (F1) pos=0.66 in=90 out=-120 w=2;
        
        \marked (B1) to (F1) pos=0.2 in=-120 out=90 ;
        \marked (B2) to (F1) pos=0.4 in=-90 out=90 ;
        \marked (B3) to (F1) pos=0.8 in=-60 out=90 ;
        \marked (B4) to (F2) pos=0.25 in=-90 out=90 ;
        \marked (B5) to (F2) pos=0.9 in=-60 out=90 ;
    \end{tikzpicture} & & \begin{tikzpicture}[line cap=round,line join=round,x=0.75cm,y=0.75cm]
        % sommets
        \floor (F1) at (-0.5,0);
        \floor (F2) at (0,2);
        \floor (F3) at (0,4);
        \floor (F4) at (-0.5,6);
        \node (T1) at (-1,8) {};
        \node (T2) at (0,8) {};
        \node (T3) at (1,8) {};
        \node (B1) at (-1,-2) {};
        \node (B2) at (0,-2) {};
        \node (B3) at (1,-2) {};
        % arêtes
        \marked (T1) to (F4) pos=0.2 in=120 out=-90 ;
        \marked (T2) to (F4) pos=0.5 in=60 out=-90 ;
        \marked (T3) to (F3) pos=0.4 in=60 out=-90 ;

        \leftmarked (F4) to (F3) pos=0.5 in=120 out=-90 w=2 ;
        \rightmarked (F3) to (F2) pos=0.8 in=60 out=-60 w=2 ;
        \marked (F3) to (F2) pos=0.2 in=120 out=-120 ;
        \leftmarked (F2) to (F1) pos=0.2 in=90 out=-120 w=2 ;
        
        \marked (B1) to (F1) pos=0.3 in=-120 out=90 ;
        \marked (B2) to (F1) pos=0.7 in=-60 out=90 ;
        \marked (B3) to (F2) pos=0.65 in=-60 out=90 ;
    \end{tikzpicture} & &  \begin{tikzpicture}[line cap=round,line join=round,x=0.75cm,y=0.75cm]
        % sommets
        \floor (F1) at (-0.5,0);
        \floor (F2) at (0,2);
        \floor (F3) at (0,4);
        \floor (F4) at (-0.5,6);
        \node (T1) at (-1,8) {};
        \node (T2) at (0,8) {};
        \node (T3) at (1,8) {};
        \node (B1) at (-1,-2) {};
        \node (B2) at (0,-2) {};
        \node (B3) at (1,-2) {};
        % arêtes
        \marked (T1) to (F4) pos=0.2 in=120 out=-90 ;
        \marked (T2) to (F4) pos=0.5 in=60 out=-90 ;
        \marked (T3) to (F3) pos=0.4 in=60 out=-90 ;

        \leftmarked (F4) to (F3) pos=0.5 in=120 out=-90 w=2 ;
        \marked (F3) to (F2) pos=0.8 in=60 out=-60;
        \draw (0.8,3) node {$2$} ;
        \doublemarked (F3) to (F2) pos=0.2 in=120 out=-120 ;
        \leftmarked (F2) to (F1) pos=0.2 in=90 out=-120 w=2 ;
        
        \marked (B1) to (F1) pos=0.3 in=-120 out=90 ;
        \marked (B2) to (F1) pos=0.7 in=-60 out=90 ;
        \marked (B3) to (F2) pos=0.65 in=-60 out=90 ;
    \end{tikzpicture} \\
    $\D_1$ & $\ \ \ $ & $\D_2$ & $\ \ \ $ & $\D_3$ \\
    \end{tabular}
    \caption{Nerved diagrams of genus $0$ and $1$.}
    \label{fig-expl-nerved-diagram}
\end{figure}
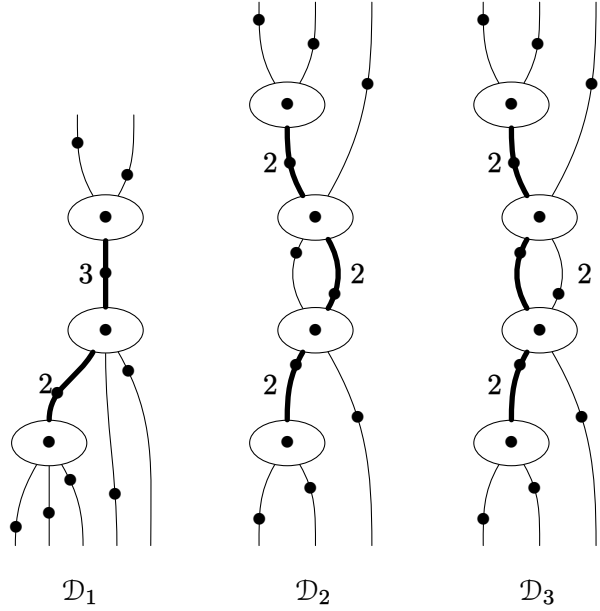

        \begin{expl}
            Assume we chose $M=1$, so that there are no condition on the weight of the edges on the nerve. On Figure \ref{fig-expl-nerved-diagram} we depict three nerved diagrams, the nerve consists in thickened edges. The first nerved diagram is the unique nerved diagram associated to the underlying genus $0$ diagram. The remaining two nerved diagrams have the same underlying genus $1$ diagram. If we had taken $M=2$, only one of the two edges between the second and third floor could have been chosen in the nerve.
        \end{expl}
        
We assign to each nerved diagram a multiplicity so that the count of nerved diagram matches the count of diagrams.

\begin{defi}
    Let $\D$ be diagram of genus $g$ in the class $aE+bF$ and assume $b>i+(g+1)M$. Let $N(\D)$ be the number of nerves of $\D$. The \emph{multiplicity} of a nerved diagram $\tilde\D$ is $\mu(\tD)=\frac{1}{N(\D)}\mu(\D)$.
\end{defi}

        \begin{rem}
            Forgetting about the ends of the diagram, a nerve is a spanning tree of the  underlying graph so that there are $g$ bounded edges not belonging to the nerve.
        \end{rem}

        \medskip

\subsubsection{Constructing genus $1$ nerved diagrams from genus $0$ ones.}

Let $\tDDD_g$ be the set of nerved marked diagram of genus $g$ in the class $aE+bF$. Assume $b>i+2M$. We have a map
        $$\mathrm{ft}:\tDDD_1\to\tDDD_0$$
that forgets the unique bounded edge $e$ not on the nerve and add $w(e)$ to the weights of all the edges between the two vertices to which $e$ was attached. Conversely, we can construct a genus $1$ nerved marked diagram from a genus $0$ one by adding an edge $e$ with weight $w$, and removing $w$ to the weights of all the edges between the two vertices to which $e$ is attached.  This is possible if we are provided with the weight $w$ of the added edge, the place of its marking between two floors $m$ and $m+1$, and the floors it is attached to, encoded by a pair $(s_+,s_-)$ that are the numbers of floors it skips above and below its marking. This data is subject to the following constraints:
            \begin{itemize}
                \item $s_-\leqslant m-1$ and $s_+\leqslant a-m-1$, 
                
                \item  $w\leqslant \min (w(e))-M$, where the minimum is over the weights of the edges of the nerve between the floors $m-s_-$ and $m+1+s_+$, so that the weights of the nerves are still $\geqslant M$.
            \end{itemize}

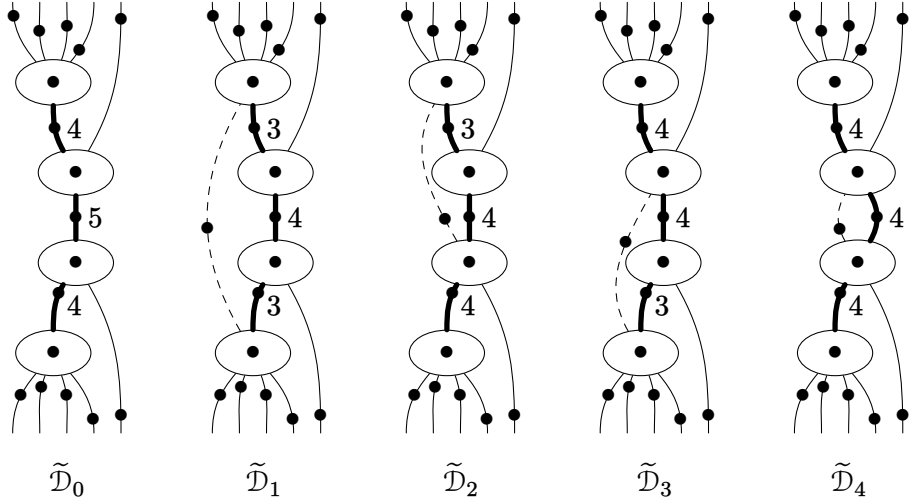
\begin{figure}[h!]
    \centering
    \begin{tabular}{ccccccccc}
    \begin{tikzpicture}[line cap=round,line join=round,x=0.75cm,y=0.75cm,scale=0.8]
        % sommets
        \floor (F1) at (-0.5,0);
        \floor (F2) at (0,2);
        \floor (F3) at (0,4);
        \floor (F4) at (-0.5,6);
        \node (T1) at (-1.4,8) {};
        \node (T2) at (-0.8,8) {};
        \node (T3) at (-0.2,8) {};
        \node (T4) at (0.4,8) {};
        \node (T5) at (1,8) {};
        \node (B1) at (-1.4,-2) {};
        \node (B2) at (-0.8,-2) {};
        \node (B3) at (-0.2,-2) {};
        \node (B4) at (0.4,-2) {};
        \node (B5) at (1,-2) {};
        % arêtes
        \marked (T1) to (F4) pos=0.2 in=130 out=-90 ;
        \marked (T2) to (F4) pos=0.5 in=110 out=-90 ;
        \marked (T3) to (F4) pos=0.4 in=70 out=-90 ;
        \marked (T4) to (F4) pos=0.8 in=50 out=-90 ;
        \marked (T5) to (F3) pos=0.1 in=60 out=-90 ;

        \rightmarked (F4) to (F3) pos=0.5 in=120 out=-90 w=4;
        \rightmarked (F3) to (F2) pos=0.5 in=90 out=-90 w=5;
        \rightmarked (F2) to (F1) pos=0.2 in=90 out=-120 w=4;
        
        \marked (B1) to (F1) pos=0.6 in=-130 out=90 ;
        \marked (B2) to (F1) pos=0.8 in=-110 out=90 ;
        \marked (B3) to (F1) pos=0.65 in=-70 out=90 ;
        \marked (B4) to (F1) pos=0.2 in=-50 out=90 ;
        \marked (B5) to (F2) pos=0.1 in=-60 out=90 ;
    \end{tikzpicture} & & \begin{tikzpicture}[line cap=round,line join=round,x=0.75cm,y=0.75cm,scale=0.8]
        % sommets
        \floor (F1) at (-0.5,0);
        \floor (F2) at (0,2);
        \floor (F3) at (0,4);
        \floor (F4) at (-0.5,6);
        \node (T1) at (-1.4,8) {};
        \node (T2) at (-0.8,8) {};
        \node (T3) at (-0.2,8) {};
        \node (T4) at (0.4,8) {};
        \node (T5) at (1,8) {};
        \node (B1) at (-1.4,-2) {};
        \node (B2) at (-0.8,-2) {};
        \node (B3) at (-0.2,-2) {};
        \node (B4) at (0.4,-2) {};
        \node (B5) at (1,-2) {};
        % arêtes
        \marked (T1) to (F4) pos=0.2 in=130 out=-90 ;
        \marked (T2) to (F4) pos=0.5 in=110 out=-90 ;
        \marked (T3) to (F4) pos=0.4 in=70 out=-90 ;
        \marked (T4) to (F4) pos=0.8 in=50 out=-90 ;
        \marked (T5) to (F3) pos=0.1 in=60 out=-90 ;

        \rightmarked (F4) to (F3) pos=0.5 in=120 out=-90 w=3;
        \rightmarked (F3) to (F2) pos=0.5 in=90 out=-90 w=4;
        \rightmarked (F2) to (F1) pos=0.2 in=90 out=-120 w=3;

        \dashedmarked (F1) to (F4) pos=0.45 in=-120 out=120;
        
        \marked (B1) to (F1) pos=0.6 in=-130 out=90 ;
        \marked (B2) to (F1) pos=0.8 in=-110 out=90 ;
        \marked (B3) to (F1) pos=0.65 in=-70 out=90 ;
        \marked (B4) to (F1) pos=0.2 in=-50 out=90 ;
        \marked (B5) to (F2) pos=0.1 in=-60 out=90 ;
    \end{tikzpicture} & & \begin{tikzpicture}[line cap=round,line join=round,x=0.75cm,y=0.75cm,scale=0.8]
        % sommets
        \floor (F1) at (-0.5,0);
        \floor (F2) at (0,2);
        \floor (F3) at (0,4);
        \floor (F4) at (-0.5,6);
        \node (T1) at (-1.4,8) {};
        \node (T2) at (-0.8,8) {};
        \node (T3) at (-0.2,8) {};
        \node (T4) at (0.4,8) {};
        \node (T5) at (1,8) {};
        \node (B1) at (-1.4,-2) {};
        \node (B2) at (-0.8,-2) {};
        \node (B3) at (-0.2,-2) {};
        \node (B4) at (0.4,-2) {};
        \node (B5) at (1,-2) {};
        % arêtes
        \marked (T1) to (F4) pos=0.2 in=130 out=-90 ;
        \marked (T2) to (F4) pos=0.5 in=110 out=-90 ;
        \marked (T3) to (F4) pos=0.4 in=70 out=-90 ;
        \marked (T4) to (F4) pos=0.8 in=50 out=-90 ;
        \marked (T5) to (F3) pos=0.1 in=60 out=-90 ;

        \rightmarked (F4) to (F3) pos=0.5 in=120 out=-90 w=3;
        \rightmarked (F3) to (F2) pos=0.5 in=90 out=-90 w=4;
        \rightmarked (F2) to (F1) pos=0.2 in=90 out=-120 w=4;

        \dashedmarked (F2) to (F4) pos=0.15 in=-120 out=120;
        
        \marked (B1) to (F1) pos=0.6 in=-130 out=90 ;
        \marked (B2) to (F1) pos=0.8 in=-110 out=90 ;
        \marked (B3) to (F1) pos=0.65 in=-70 out=90 ;
        \marked (B4) to (F1) pos=0.2 in=-50 out=90 ;
        \marked (B5) to (F2) pos=0.1 in=-60 out=90 ;
    \end{tikzpicture} & & \begin{tikzpicture}[line cap=round,line join=round,x=0.75cm,y=0.75cm,scale=0.8]
        % sommets
        \floor (F1) at (-0.5,0);
        \floor (F2) at (0,2);
        \floor (F3) at (0,4);
        \floor (F4) at (-0.5,6);
        \node (T1) at (-1.4,8) {};
        \node (T2) at (-0.8,8) {};
        \node (T3) at (-0.2,8) {};
        \node (T4) at (0.4,8) {};
        \node (T5) at (1,8) {};
        \node (B1) at (-1.4,-2) {};
        \node (B2) at (-0.8,-2) {};
        \node (B3) at (-0.2,-2) {};
        \node (B4) at (0.4,-2) {};
        \node (B5) at (1,-2) {};
        % arêtes
        \marked (T1) to (F4) pos=0.2 in=130 out=-90 ;
        \marked (T2) to (F4) pos=0.5 in=110 out=-90 ;
        \marked (T3) to (F4) pos=0.4 in=70 out=-90 ;
        \marked (T4) to (F4) pos=0.8 in=50 out=-90 ;
        \marked (T5) to (F3) pos=0.1 in=60 out=-90 ;

        \rightmarked (F4) to (F3) pos=0.5 in=120 out=-90 w=4;
        \rightmarked (F3) to (F2) pos=0.5 in=90 out=-90 w=4;
        \rightmarked (F2) to (F1) pos=0.2 in=90 out=-120 w=3;

        \dashedmarked (F1) to (F3) pos=0.65 in=-120 out=120;
        
        \marked (B1) to (F1) pos=0.6 in=-130 out=90 ;
        \marked (B2) to (F1) pos=0.8 in=-110 out=90 ;
        \marked (B3) to (F1) pos=0.65 in=-70 out=90 ;
        \marked (B4) to (F1) pos=0.2 in=-50 out=90 ;
        \marked (B5) to (F2) pos=0.1 in=-60 out=90 ;
    \end{tikzpicture} & & \begin{tikzpicture}[line cap=round,line join=round,x=0.75cm,y=0.75cm,scale=0.8]
        % sommets
        \floor (F1) at (-0.5,0);
        \floor (F2) at (0,2);
        \floor (F3) at (0,4);
        \floor (F4) at (-0.5,6);
        \node (T1) at (-1.4,8) {};
        \node (T2) at (-0.8,8) {};
        \node (T3) at (-0.2,8) {};
        \node (T4) at (0.4,8) {};
        \node (T5) at (1,8) {};
        \node (B1) at (-1.4,-2) {};
        \node (B2) at (-0.8,-2) {};
        \node (B3) at (-0.2,-2) {};
        \node (B4) at (0.4,-2) {};
        \node (B5) at (1,-2) {};
        % arêtes
        \marked (T1) to (F4) pos=0.2 in=130 out=-90 ;
        \marked (T2) to (F4) pos=0.5 in=110 out=-90 ;
        \marked (T3) to (F4) pos=0.4 in=70 out=-90 ;
        \marked (T4) to (F4) pos=0.8 in=50 out=-90 ;
        \marked (T5) to (F3) pos=0.1 in=60 out=-90 ;

        \rightmarked (F4) to (F3) pos=0.5 in=120 out=-90 w=4;
        \rightmarked (F3) to (F2) pos=0.5 in=60 out=-60 w=4;
        \rightmarked (F2) to (F1) pos=0.2 in=90 out=-120 w=4;

        \dashedmarked (F2) to (F3) pos=0.25 in=-120 out=120;
        
        \marked (B1) to (F1) pos=0.6 in=-130 out=90 ;
        \marked (B2) to (F1) pos=0.8 in=-110 out=90 ;
        \marked (B3) to (F1) pos=0.65 in=-70 out=90 ;
        \marked (B4) to (F1) pos=0.2 in=-50 out=90 ;
        \marked (B5) to (F2) pos=0.1 in=-60 out=90 ;
    \end{tikzpicture} \\
    $\tD_0$ & & $\tD_1$ & & $\tD_2$ & & $\tD_3$ & & $\tD_4$
    \end{tabular}
    \caption{On the left a genus $0$ nerved marked diagram. On the right, various genus $1$ nerved marked diagrams that we can obtain by adding an edge of weight $1$ with a marking between the second and third floor.} 
    \label{fig-constr-nerved-diagrams}
\end{figure}

        \begin{expl}
            Assume $M=2$. On Figure \ref{fig-constr-nerved-diagrams} we depicted various ways to get a genus $1$ nerved marked diagram by adding a dashed edge of weight $1$ to $\tD_0$. If we have $(s_+,s_-)=(1,1)$, we get $\tD_1$ because the edge skips one floor above its marking, and one below. Taking $(1,0)$ or $(0,1)$ instead, we get $\tD_2$ and $\tD_3$. If $s_+=s_-=0$, we get $\tD_4$. In all these examples we choose the marking of the added edge to be between the one of the second floor and the one of the bounded edge between the second and third floor.
            
            Let us try to increase the weight $w$ of the dashed edge. For $\tD_1,\tD_2,\tD_3$, $w$ can also be set equal to $2$, but not $3$ since one of the edges on the nerve would get weight $1<M$. For $\tD_4$, we can take $w=2$ or $3$, and in that case the underlying diagram has two possible nerves.
        \end{expl}

        We now relate the multiplicity of a nerved diagram constructed by the above process to the multiplicity of the initial genus $0$ diagram.

        \begin{lem}\label{lem-relation-mult-nerved-diagrams}
            Assume $M>i$ and $b>i+2M$. Let $\tD$ be a genus $0$ nerved marked diagram with $\codeg(\tD)\leqslant i$ in the class $aE+bF$, and let $\tE$ be the genus $1$ marked nerved diagram constructed by the data of the position of the marking, weight $w$ and $(s_+,s_-)$. Then we have
            $$\mu(\tE)=\frac{1}{1+\mathds{1}_{w\geqslant M}}(1-x^w)^2x^{w(s_++s_-)}\mu(\tD) \modulo x^{i+1}.$$
        \end{lem}

        \begin{proof}
        If $w\geq M$ then $s_+=s_-=0$, otherwise the codegree would be greater than $M$, and also $i$. Hence in that case the added edge links two consecutives floors and, there are two possible nerves. If $w<M$, there is a unique nerve. Hence one has $N(\E)=1+\mathds{1}_{w\geqslant M}$.
        
       By Lemma \ref{lem-poids-grands}, the hypothesis ensures that the sum of weights between consecutive floors in $\D$ is bigger than $2M$. Thus, the only edge potentially contributing to the multiplicity of $\E$ is the one we add, yielding a factor $(1-x^w)^2$. The codegree this edge provides is $w(s_++s_-)$ since it has weight $w$ and skips exactly $s_++s_-$ floors. As the weights of the edges in the nerve are still bigger than $M$ after we added the new edge, they still do not contribute to the multiplicity modulo $x^{i+1}$.
        \end{proof}

        Conversely, we can add the multiplicities of the genus $1$ nerved marked diagrams constructed from a genus $0$ nerved marked diagram. Let $\ang{m}=\frac{x^m}{1-x^m}$.

        \begin{lem}\label{lem-function-to-integrate-to get-genus-1}
            Let $\tD$ be a genus $0$ nerved marked diagram and let $1\leqslant m\leqslant a-1$. Assume $M>i$, $a>2i$ and $b>i+2M$. Let $\mathrm{pos}_m$ be the number of positions where to insert a marking between the floors $m$ and $m+1$. Let $\tomega_m$ be the weight of the edge between these floors. The sum of multiplicities of genus $1$ nerved marked diagrams obtained by inserting an edge with a marking between these floors is
            $$\mathrm{pos}_m\cdot \left(\frac{\tomega_m-1}{2} - d_m \right)\mu(\tD), \text{ where }d_m =\left\{\begin{array}{ll}
                \ang{m} & \text{ if }  m\leqslant i, \\
                \ang{a-m} & \text{ if }  m\geqslant a-i, \\
                0 & \text{ else.} 
            \end{array} \right. .$$
        \end{lem}

        \begin{proof}
            We first choose one of the $\mathrm{pos}_m$ possible positions for the marking. We then sum over the possible choices of $w,s_\pm$.  There are two possibilities.
                \begin{itemize}
                    \item If $s_++s_->0$, we can assume the weight $w$ is bounded by $i$ since otherwise, we get multiplicity $0$ modulo $x^{i+1}$.
                    \item  If $s_+=s_-=0$, the weight $w$ may takes values from $1$ to $\tomega_m-M$, since the nerve has to keep a weight bigger than $M$. Furthermore, we start having a factor $\frac{1}{2}$ for the choices of nerves when $w\geqslant M$. For such a $w$, we have $(1-x^w)^2\equiv 1\modulo x^{i+1}$ since $M>2i\geq i$.
                \end{itemize}
            Thus by  Lemma \ref{lem-relation-mult-nerved-diagrams} we have to compute the following:
            \begin{align*}
                 & \sum_{w=1}^{i}\sum_{s_++s_->0}(1-x^w)^2x^{w(s_++s_-)} + \sum_{w=1}^{M-1}(1-x^w)^2+\sum_{w=M}^{\tomega_m-M}\frac{1}{2} \modulo x^{i+1} .
            \end{align*}
            To compute the first sum, we may add the values for $w$ going from $i+1$ to infinity since they contribute 0 modulo $x^{i+1}$. If $m\leq i$ we have the bound $s_-\leqslant m-1$, but $s_+$ can goes to $\infty$ since the excess terms contribute 0 modulo $x^{i+1}$. In that case we get for the first sum 
            \begin{align*}
                \sum_{w=1}^\infty (1-x^w)^2\left(\frac{1-x^{mw}}{(1-x^w)^2}-1\right) = & \sum_{w=1}^\infty \left[1-(1-x^w)^2 -x^{mw}\right] \\
                = & \sum_{w=1}^\infty \left[1-(1-x^w)^2\right] - \ang{m}.
            \end{align*}
            If $m \geq a-i$ we have the bound $s_+ \leq a-m-1$, but $s_-$ can goes to $\infty$ and the first sum gives
            \[ \sum_{w=1}^\infty \left[1-(1-x^w)^2\right] - \ang{a-m} .\]
            If $i<m<a-i$ then both $s_-$ and $s_+$ can go to $\infty$ so the first sum is
            \[ \sum_{w=1}^\infty \left[1-(1-x^w)^2\right] .\]
            The others two sums are
            \begin{align*}
              \sum_{w=1}^{M-1}(1-x^w)^2+\sum_{w=M}^{\tomega_m-M}\frac{1}{2} = & \sum_{w=1}^{M-1}\left[(1-x^w)^2-1\right] +M-1+\frac{\tomega_m-M-M+1}{2}  \\
              \equiv & \sum_{w=1}^\infty \left[(1-x^w)^2-1\right] +\frac{\tomega_m-1}{2} \modulo x^{i+1}.
            \end{align*}
            Putting all sums together, the two sums over $w$ cancel and we get the result.
        \end{proof}

\subsubsection{Integration over the space of genus $0$ diagrams.} In the computation of the genus $0$ asymptotic refined invariants, we encoded marked diagrams with words and proved that the set of words is in bijection with a subset of $\S(b)\times\S(b+\delta a)$. Elements of $\S(n)$ were assigned multiplicities
\[ \mu_{\S(n)}(\sfk)=(1-x)^n x^{\codeg(\sfk)} . \]

Recall that we have
maps $\ell_0, \ell_j^{(k)} : \S \to \ZZ_{\geq0}$ that give the lengths of the words of a sentence. Let $\L$ be the lengths space, \ie the space of non-negative integer sequences $(l_j^{(k)})_{j,k}$ with finite support, and $\pi$ be the map $\pi = (\ell_j^{(k)})_{j,k} : \S(n) \to \L$ that maps a sentence to the lengths of its words except the first one. 
To each element $\bfl = (l_j^{(k)})_{j,k} \in\L$, we assign a weight $\mu_\L(\bfl)=\prod_{j,k}x^{jl_j^{(k)}}$.
        
Formally, it is possible to see $\mu_{\S(n)}$ and $\mu_\L$ as measures on their corresponding domain, which are discrete spaces. These measures have values in the quotient ring $\ZZ[x]/(x^{i+1})$ for our choice of $i$. From this point of view, weighted sums become integrals. Moreover, this integral is $\ZZ[x]/(x^{i+1})$-linear. There are several reasons for such a consideration: it shortens notations, it becomes easier to see some computational steps, and it formalizes the deletion of diagrams with zero weight. The idea to compute the asymptotic refined invariant in genus $1$ is now to integrate the function given by Lemma \ref{lem-function-to-integrate-to get-genus-1} on the space of genus 0 diagrams.

Lemma \ref{lem-generating-series-T-words} states that $\mu_{\S(n)}$ and $\mu_\L$ have total weight $p(x)^2$, so that we may consider the normalized measures $\nu_{\S(n)}=\frac{1}{p(x)^2}\mu_{\S(n)}$ and $\nu_\L=\frac{1}{p(x)^2}\mu_\L$. For product spaces, we consider the product measures.
During the proof of Lemma \ref{lem-generating-series-T-words}, we have
$$\sum_{\substack{\sfk\in\S(n) \\ \pi(\sfk)=\bfl}} \mu_{\S(n)}(\sfk) = \mu_{\S(n)}(\pi^{-1}(\bfl)) = \prod_{j,k}x^{jl_j^{(k)}} = \mu_\L(\bfl).$$

\subsubsection{Some integral computations.} Before going through the main computation, we introduce some functions on $\L$ and $\S(n)$, and compute their integrals against the normalized measures. Consider first the lengths functions $\ell_j^{(k)}$, which are the coordinate functions on $\L$.

\begin{lem}
We have the following integrals: 
                $$\int_\L \ell_m^{(r)}\dd\nu_\L = \ang{m} , \ 
                    \int_\L (\ell_m^{(r)})^2\dd\nu_\L = \ang{m}+2\ang{m}^2 $$
                with $\ang{m}=\frac{x^m}{1-x^m}$.    
\end{lem}

\begin{proof}
Indeed, by definition, we have
    \begin{align*}
    \int_\L \ell_m^{(r)}\dd\nu_\L &= \frac{1}{p(x)^2}\sum_{\bfl\in\L} l_m^{(r)}\prod_{j,k} x^{jl_j^{(k)}} \\
    &= \frac{1}{p(x)^2}\left(\sum_{l_m^{(r)}=0}^\infty l_m^{(r)}x^{ml_m^{(r)}}\right)\prod_{(j,k)\neq (m,r)}\left(\sum_{l_j^{(k)}=0}^\infty x^{jl_j^{(k)}}\right).
    \end{align*}
    We then use the identity $\sum_{\alpha=0}^\infty \alpha y^\alpha = \frac{y}{(1-y)^2}$. For the second integral, we use $\sum_{\alpha=0}^\infty \alpha^2 y^\alpha = \frac{y+y^2}{(1-y)^3}$.
\end{proof}

In fact, this method, which is an analog of Fubini's theorem, works for computing the integral of any monomial in the $\ell_j^{(k)}$: the integral of a monomial is equal to the product of integrals over each of the variables appearing in the monomial. Hence, it reduces down to the computation of the sums $\sum_{\alpha=0}^\infty \alpha^r y^\alpha$.

We then set $\ell_m=\ell_m^{(1)}+\ell_m^{(2)}$, so that we now have
        $$\int_\L \ell_m\dd\nu_\L = 2\ang{m} \text{ and }\int_\L \ell_m^2\dd\nu_\L = 2\ang{m}+6\ang{m}^2.$$
In particular, the following affine function 
\[ e_m=(1-x^m)\frac{\ell_m+2}{2} \]
defined on $\L$ has integral equal to $1$.

By composing with $\pi:\S(n)\to\L$, it is possible to pull-back functions on $\L$ to get functions on $\S(n)$. Due to the normalization by the total weight, their integrals are preserved.

\begin{defi}
We define on $\S(n)$ the \emph{leak function} $\phi_m[n](\sfk)$ equal to the number of letters with an index bigger than $m$. To get a function of $\bfl\in\L$, we average over the set $\pi^{-1}(\bfl)$ of sentences with lengths $\bfl$:
$$\varphi_m[n](\bfl)=\frac{1}{1-x^m}\frac{1}{\mu_\L(\bfl)}\int_{\pi^{-1}(\bfl)}\phi_m[n]\dd\mu_{\S(n)}.$$
\end{defi}

Lemma \ref{lem-expr-leak-function-in-l-monomials} expresses the function $\varphi_m[n](\bfl)$ in terms of the monomials $\ell_j$ on $\L$.

    \begin{rem}
        On the diagram side, the leak function $\phi_m$ corresponds to the number of ends skipping the floor $m$. It is also equal to $\omega_m-\tomega_m$, which is the complement of the weight between the floors $m$ and $m+1$ to the maximal possible weight $\omega_m=b+\delta m$.
    \end{rem}

\begin{lem}\label{lem-expr-leak-function-in-l-monomials}
     We have the following expressions on $\L$:
     $$\varphi_m[n](\bfl) = n \ang{m}+\psi_m(\bfl), \text{ where } \psi_m = \ang{m}\sum_{j=1}^m \frac{\ell_j}{\ang{j}}+\sum_{j=m+1}^\infty \ell_j. $$
     \end{lem}

\begin{proof}
    Let $\bfl\in\L$ and  $\sfk=(\sfS_0,(\sfS_j^{(k)})_{j,k}) \in \pi^{-1}(\bfl)$ be a sentence.
    In terms of the letters, the leak function $\phi_m[n]$ is
    $$\phi_m[n](\sfk)= \sum_{\sfs\in\sfS_0}\mathds{1}(p\geqslant m \text{ with }\sfs=\sfs_p) + \dsum_{j,k}\sum_{\sfs\in\sfS_j^{(k)}}\mathds{1}(p\geqslant m \text{ with }\sfs=\sfs_p).$$
    Indeed, the leak is due to the ends that skip the floor $m$, \ie the letters $\sfs_p$ with an index $p\geqslant m$. 
    We need to compute $\frac{1}{\mu_\L(\bfl)}\int_{\pi^{-1}(\bfl)}\mathds{1}(p\geqslant m \text{ with }\sfs=\sfs_p)\dd\mu_{\S(n)}$, for each term $\mathds{1}(p\geqslant m \text{ with }\sfs=\sfs_p)$  corresponding to a position of the letter $\sfs$ in the word $\sfS_j^{(k)}$. To do so, we proceed as in Lemma \ref{lem-generating-series-T-words}.
    At each letter position $\sfs'$ in $\sfS_{j'}^{(k')}$ except the one corresponding to $\sfs$, the sum over the possible values of the letter is the geometric series
        $$\sum_{p=j'}^\infty x^p = \frac{x^{j'}}{1-x}. $$
    For the position corresponding to $\sfs$, because of the condition $(p\ge m \text{ with } \sfs=\sfs_p)$ we have instead
        $$\sum_{p=j}^\infty \mathds{1}(p\geqslant m)x^p = \frac{x^{\max(j,m)}}{1-x}=x^{(m-j)_+}\frac{x^j}{1-x}, $$
    where $(m-j)_+=\max(m-j,0)$. As in Lemma \ref{lem-generating-series-T-words}, we conclude by making the product over all letter positions and we get
    \[ \int_{\pi^{-1}(\bfl)}\mathds{1}(p\geqslant m \text{ with }\sfs=\sfs_p)\dd\mu_{\S(n)} = x^{(m-j)_+}\prod_{j',k'}x^{j'l_{j'}^{(k')}} =x^{(m-j)_+}\mu_\L(\bfl) . \]
        
    Adding the above over all the letter positions in the word, we get
            
    \begin{align*}
        \frac{1}{\mu_{\L}(\bfl)}\int_{\pi^{-1}(\bfl)}\phi_m[n]\dd\nu_{\S(n)} &= \ell_0 x^m + \sum_{j=1}^m (\ell_j^{(1)}+\ell_j^{(2)})x^{m-j}+\sum_{j=m+1}^\infty (\ell_j^{(1)}+\ell_j^{(2)}) \\
        &= \left(n-\sum_{j=1}^\infty \ell_j\right)x^m +\sum_{j=1}^m \ell_j x^{m-j} + \sum_{j=m+1}^\infty \ell_j \\
        &= n x^m + \sum_{j=1}^m \ell_j(x^{m-j}-x^m) +(1-x^m)\sum_{j=m+1}^\infty \ell_j \\
        &= (1-x^m)\left[ n\ang{m} + \ang{m}\sum_{j=1}^m \frac{\ell_j}{\ang{j}}+\sum_{j=m+1}^\infty \ell_j \right].
    \end{align*}
    \end{proof}

    \begin{lem}
 We have the following integral:
      $$\int_\L e_m\psi_m \dd\nu_\L = (2m+1)\ang{m}+2\sum_{j=m+1}^\infty\ang{j}.$$
    \end{lem}

            \begin{proof}
                We use the expression of $\psi_m$ in terms of the $\ell_j$, and the following computations:
                \[ \int_\L e_m\ell_j \dd\nu_\L=\left\{ \begin{array}{ll}
                   2\ang{j}  & \text{ if } j\neq m, \\
                   3\ang{m}  & \text{ if } j=m. 
                \end{array}\right. \]
                Hence, $\int e_m \ell_j=\int \ell_j$ except for $j=m$, where we add $\ang{m}$. Thus, we have
                \begin{align*}
                  \int_\L e_m\psi_m & = \int_\L\psi_m+\ang{m} \\
                  & =  \ang{m}\sum_{j=1}^m \frac{2\ang{j}}{\ang{j}}+\sum_{m+1}^\infty 2\ang{j} +\ang{m} \\
                  & = (2m+1)\ang{m}+2\sum_{j=m+1}^\infty\ang{j}.
                \end{align*}
            \end{proof}

So far, we defined functions on $\S(n)$. The genus $0$ marked diagrams of codegree smaller than $i$ are in bijection with a subset of $\S(b)\times\S(b+\delta a)$. By definition, the complement of this subset has measure $0$ since it consists of elements with codegree strictly bigger than $i$. Let $\rho_1$, $\rho_2$ be the projections of $\S(b)\times\S(b+\delta a)$ to $\S(b)$ and $\S(b+\delta a)$. We can thus pull-back functions by $\rho_1$ and $\rho_2$ and obtain the following functions.
\begin{itemize}
    \item The number $\mathrm{pos}_m$ of positions for a marking between the floors $m$ and $m+1$ is:
        \begin{itemize}[label=$\circ$]
            \item equal to $\rho_1^*\ell_m+2=\rho_1^*\ell_m^{(1)}+\rho_1^*\ell_m^{(2)}+2$ if $m \leq i$, where $\ell_m$ is pull-back from $\S(b)$,
            \item equal to $2$ if $i < m < a-i$, since the length functions are $0$,
            \item equal to $\rho_2^*\ell_{a-m}+2$ if $m \geq a-i$, where $\ell_{a-m}$ is now pull-back from $\S(b+\delta a)$ instead of $\S(b)$.
         \end{itemize}
    \item We have the same phenomenon for the leak function on a diagram: for $m \leq i$, it is the pull-back of the leak function on $\S(b)$, then it is $0$ for $i < m < a-i$, and gets pulled-back from $\S(b+\delta a)$ for $m \geq a-i$.
\end{itemize}

\subsubsection{Computation of the asymptotic refined invariant}

For $n$ a positive integer, we consider the function $\sigma_1(n) = \sum_{d|n}d$, and its generating series
\[ E_2(x) = \dsum_{n\geq1} \sigma_1(n) x^n . \]
            
\begin{lem}\label{lem-computation-eisenstein-series}
    One has
    \[ E_2(x) = \sum_{n=1}^\infty n\frac{x^n}{1-x^n}=\sum_{n=1}^\infty\frac{x^n}{(1-x^n)^2}=\sum_{n=1}^\infty\sum_{j=n}^\infty\frac{x^j}{1-x^j} . \]
\end{lem}

\begin{proof}
    Expanding $\frac{1}{1-x^n}$ yields the first two expressions for $E_2(x)$. The last expression yields the first one when switching the sums over $n$ and $j$.
\end{proof}

\begin{theo}\label{theo-AR-genus1-hirzebruch}
    The genus $1$ asymptotic refined invariant of Hirzebruch surfaces is given by
        $$AR_1^{\FF_\delta} = p(x)^4\left(g_{\max}-12 E_2(x) \right),$$
    where $g_{\max}=\frac{1}{2}(a-1)(2b+\delta a-2)$ is the genus of a smooth curve in the class $aE+bF$.
\end{theo}

\begin{proof}
The computation of the asymptotic refined invariant goes through three steps: expressing then integrating over $\S(b)\times\S(b+\delta a)$ the function from Lemma \ref{lem-function-to-integrate-to get-genus-1}, and summing these integrals over $m$ from $1$ to $a-1$.
                
$\circ$ \textbf{First step: expression over $\S(b)\times\S(b+\delta a)$.} By Lemma \ref{lem-function-to-integrate-to get-genus-1}, the function giving the sum of multiplicities for insertion of a marking between the floors $m$ and $m+1$ is given by $\mathrm{pos}_m\left( \frac{\tomega_m-1}{2}-d_m\right)$ and has the following values:
\[ \left\{ \begin{array}{ll}
        (\rho_1^*\ell_m+2)\left(\dfrac{\omega_m-\rho_1^*\phi_m[b]-1}{2}-\ang{m}\right)  & \text{ if }m\leq i, \\
        \omega_m-1  &  \text{ if }i<m<a-i \\
        (\rho_2^*\ell_{a-m}+2)\left(\dfrac{\omega_m-\rho_2^*\phi_{a-m}[b+\delta a]-1}{2}-\ang{a-m}\right) & \text{ if }m\geq a-i,
    \end{array} \right.\]
where in the first (resp. last) row, functions are pull-back from $\S(b)$ (resp. $\S(b+\delta a)$).
For each value of $m$, we now need to integrate the above function, and then sum over $1\leq m\leq a-1$.

\medskip

 $\circ$ \textbf{Second step: integration over $\S(b)\times\S(b+\delta a)$.} If $i<m<a-i$, we have
    $$\int_{\S(b)\times\S(b+\delta a)}(\omega_m-1)\dd \nu=\omega_m-1.$$
                
Assume now that $m\leqslant i$. Since $\int_{\S(b+\delta a)}1=1$ we have
    \begin{align*}
     & \int_{\S(b)\times\S(b+\delta a)}(\rho_1^*\ell_m+2)\left(\dfrac{\omega_m-\rho_1^*\phi_m[b]-1}{2}-\ang{m}\right) \\
     = & \int_{\S(b)}(\ell_m+2)\left(\dfrac{\omega_m-\phi_m[b]-1}{2}-\ang{m}\right).
     \end{align*}
Recall that we set $e_m=(1-x^m)\frac{\ell_m+2}{2}$ so the integrand rewrites
\begin{align*}
    & e_m(\omega_m-1) + e_m\left( (\omega_m-1)\ang{m}-\frac{\phi_m[b]}{1-x^m}-2\frac{x^m}{(1-x^m)^2}\right). 
\end{align*} 
To compute the integral over $\S(b)$, we first regroup over each $\pi^{-1}(\bfl)$ considering $\frac{1}{\mu_\L(\bfl)}\int_{\pi^{-1}(\bfl)} $. This way, we get a function to integrate over $\L$. This function is
    $$e_m (\omega_m-1) + e_m\left((\omega_m-1)\ang{m}-\varphi_m[b] -2\frac{x^m}{(1-x^m)^2}\right) .$$
Because $\int_\L e_m =1$ we have $\int_\L e_m(\omega_m-1)=\omega_m-1$. It remains to compute the integral of the correction term $ e_m\left((\omega_m-1)\ang{m}-\varphi_m[b] -2\frac{x^m}{(1-x^m)^2}\right)$. As $m$ is close to $1$ one has:
    \[ \varphi_m[b] = b\ang{m}+\psi_m \ \text{ and } \ \omega_m = b+\delta m . \] 
We finally get 
    \begin{align*}
    &\ \int_\L e_m\left( (\omega_m-1)\ang{m} - \varphi_m[b] -2\frac{x^m}{(1-x^m)^2}\right) d \nu_\L  \\
        = &\ \int_\L e_m\left( (b+\delta m-1)\ang{m} - b\ang{m}-\psi_m -2\frac{x^m}{(1-x^m)^2}\right) d \nu_\L  \\
        = &\ \int_\L e_m \left( (\delta m-1)\ang{m}-\psi_m -2\frac{x^m}{(1-x^m)^2}\right) d \nu_\L  \\
        = &\ (\delta m-1)\ang{m}  - (2m+1)\ang{m}-2\sum_{j=m+1}^\infty \ang{j} - 2\frac{x^m}{(1-x^m)^2}  \\
        = &\ (\delta-2)m\frac{x^m}{1-x^m} -2\sum_{j=m}^\infty\ang{j} -2\frac{x^m}{(1-x^m)^2} 
    \end{align*}
\ie 
    \begin{align*} 
     &\ \int_{\S(b)\times\S(b+\delta a)} \mathrm{pos}_m\left( \frac{\tomega_m-1}{2}-d_m\right) \\
     = &\ \omega_m-1 + (\delta-2)m\frac{x^m}{1-x^m} -2\sum_{j=m}^\infty\ang{j} -2\frac{x^m}{(1-x^m)^2} .
     \end{align*}
                
If $m \geq a-i$, with $m'=a-m$ similar computations lead to
    \begin{align*} 
    &\ \int_{\S(b)\times\S(b+\delta a)} \mathrm{pos}_m\left( \frac{\tomega_m-1}{2}-d_m\right) \\
    =&\ \omega_m-1 - (\delta+2)m'\frac{x^{m'}}{1-x^{m'}}-2\sum_{j=m'}^\infty \ang{j} -2\frac{x^{m'}}{(1-x^{m'})^2}.
    \end{align*}

\medskip

$\circ$ \textbf{Third step: summation over the values of $m$.} We have several sums to compute.
\begin{itemize}
    \item Whatever the value of $m$ is, the term $\omega_m-1$ appears. We need to sum these terms, and one has
$$\sum_{m=1}^{a-1}(\omega_m-1)=g_{\max},$$
since it is the number of interior lattice points of the associated Newton polygon.
                
     \item We have to sum the correction terms for $1\leq m\leq i$. Since the formula for the correction term gives $0$ modulo $x^{i+1}$ when $m>i$, we let $m$ goes to $\infty$. By Lemma \ref{lem-computation-eisenstein-series}, the sum of the correction terms is
         $$(\delta -2)E_2(x)-2E_2(x)-2E_2(x)=(\delta-6)E_2(x).$$
                
    \item For the correction terms for $a-i \leq m \leq a-1$, with $m'=a-m$ we sum over $m'$ going from $1$ to $\infty$ and get
                $$ -(\delta+6)E_2(x).$$
                    \end{itemize}
                Adding the three contributions, we obtained $g_{\max}-12E_2(x)$. Multiplying by the total weight of the space $p(x)^4$ finishes the computation.
\end{proof}

\subsection{The case of $h$-transverse toric surfaces.}

The computations made in the Hirzebruch case remain valid with two differences. First, we now need to take into account the sloping pairs of the floors. The marked diagrams of genus 0 and codegree at most $i$ are in bijection with a subset of $\S(b^\bot)\times\S(b^\top)\times\P^{\chi-4}$, where $\chi$ is the number of corners of the polygon and $\P$, defined in Section \ref{sec-definition-pearl}, encodes the default of growth of the slopes. Second, the self-intersection of the divisors corresponding to the top and bottom horizontal sides, equal to $\delta$ and $-\delta$ in the Hirzebruch case, are not opposite anymore, see Lemma \ref{lem-toric-surface-euler-char-self-intersection}.

    \begin{lem}\label{lem-generating-series-corners-g1case}
        We have the following generating series:
        $$\sum_{\pfk\in\P}\codeg(\pfk) x^{\codeg(\pfk)}= E_2(x)p(x).$$
    \end{lem}

    \begin{proof}
        We computed in Lemma \ref{lem-generating-series-corners} the generating series of $x^{\codeg(\pfk)}$, so we just need to differentiate the relation, multiply by $x$ and use Lemma \ref{lem-computation-eisenstein-series}:
        \begin{align*}
            x\frac{\mathrm{d}}{\mathrm{d}x}\prod_{j=1}^\infty \frac{1}{1-x^j} &= \sum_{m=1}^\infty m\frac{x^m}{(1-x^m)^2}\prod_{j\neq m}\frac{1}{1-x^j} \\
            &= \left( \sum_{m=1}^\infty m\frac{x^m}{1-x^m} \right)\prod_{j=1}^\infty\frac{1}{1-x^j} \\
            &= E_2(x)p(x). 
        \end{align*}
    \end{proof}

\begin{theo}\label{theo-AR-genus1-general}
Let $X$ be a toric surface with Euler characteristic $\chi$ associated to an $h$-transverse, horizontal and non-singular polygon. Let $g_{\max}=1+\frac{\beta^2+K_X\cdot\beta}{2}$ be the polynomial function on $H_2(X,\ZZ)$ that gives the genus of a smooth curve in the class $\beta$. The genus $1$ asymptotic refined invariant is given by

\[ AR^X_1 = p(x)^\chi \left(g_{\max}-12E_2(x) \right). \]
\end{theo}

\begin{proof}
    We proceed as in the Hirzebruch case. We assume each $\beta\cdot D$ is large enough for any toric divisor $D$. Consider $\Pfk=(\pfk_c)_c \in\P^{\chi-4}$ where we choose an element of $\P$ for each corner non-adjacent to a horizontal side. For a given $\Pfk$, let $\omega^\Pfk_m$ be the maximum weight between the floors $m$ and $m+1$ in a diagram obtained with the choice of sloping pairs determined by $\Pfk$. It differs from the total weight in the codegree $0$ case $\omega_m$ in the following way: any element $\pfk\in\P$ is the product of exactly $\codeg(\pfk)$ transpositions, and each of them reduces the weight at the position of the transposition by $1$. So one has
    \begin{align*}
    \sum_{m=1}^{a-1}\left(\omega^\Pfk_m-1\right) &= \sum_{m=1}^{a-1}\left(\omega_m-1\right) -\sum_c \codeg(\pfk_c) \\
    &= g_{\max}-\sum_c\codeg(\pfk_c),
    \end{align*}     
    where the sum is indexed by the corners non-adjacent to a horizontal side.

    Let $\delta_\top$ (resp. $\delta_\bot$) be minus the self-intersection of the top (resp. bottom) toric divisor. For a fixed choice of $\Pfk=(\pfk_c)_c \in\P^{\chi-4}$, the contribution of $\Pfk$ to the asymptotic refined invariant is computed as in the Hirzebruch case and is 
    \[  x^{\sum_c\codeg(\pfk_c)} p(x)^4\left[ \sum_{m=1}^{a-1}\left(\omega^{\Pfk}_m-1\right) -(12-\delta_\top-\delta_\bot)E_2(x) \right]. \]

    We now replace the sum of weights by its expression in the $\codeg(\pfk_c)$ and sum over all the possible $\mathfrak{P}=(\pfk_c)_c$. We get modulo $x^{i+1}$:
    \[
        \sum_{\substack{\mathfrak{P}=(\pfk_c)_c \\ \codeg(\pfk_c)\leqslant i}}x^{\sum\codeg(\pfk_c)}p(x)^4  \left[ \left(g_{\max}-\sum_{c}\codeg(\pfk_c)\right)  -(12-\delta_\top-\delta_\bot)E_2(x) \right] \]
    As we only care about the sum modulo $x^{i+1}$, we may add all the elements in $\P$ since the ones with higher codegree will contribute 0. There are $\chi-4$ corners where we choose an element $\pfk\in\P$. Using Lemma \ref{lem-generating-series-corners} and \ref{lem-generating-series-corners-g1case} to compute the generating series, we get
    \begin{align*}
        p(x)^{\chi-4}p(x)^4\left[ g_{\max}-(\chi-4)E_2(x) -(12-\delta_\top-\delta_\bot)E_2(x) \right] \modulo x^{i+1}. 
    \end{align*}
    Finally, Lemma \ref{lem-toric-surface-euler-char-self-intersection} allows us to conclude.
\end{proof}

\begin{rem}
    The method can be adapted by adding two additional edges and compute the genus $2$ asymptotic refined invariant, and probably more, at the cost of lengthy computations.
\end{rem}

\appendix

\section{Extension of the results to Göttsche-Schroeter invariants}

Genus $0$ Block-Göttsche invariants $BG^X_0(\beta)(q)$ admit an extension called Göttsche-Schroeter invariants \cite{gottsche_refined_2019}, that we denote by $BG^X_0(\beta,s)(q)$. In this notation, $s$ is a parameter that takes into account how many pairs of complex conjugated points we fix when computing Welschinger invariants. Recently, Shustin and Sinichkin \cite{shustin-2024-refined} and the second author \cite{mevel-2024-combinatorial} independently showed that one can define a similar quantity $BG^X_g(\beta,s)(q)$ for any genus $g$. In this appendix we present how to adapt the proofs of the present paper with non-zero $s$. Note that the genus $0$ case was already handled in \cite{mevel2023universal}. 

Recall from \cite{mevel-2024-combinatorial} that the floor diagram recipe to compute $BG^X_g(\beta,s)(q)$ requires to choose a \textit{pairing} $S$ of order $s$. In this appendix we will take $S=\{\{1,2\},\dots,\{2s-1,2s\}\}$, and we say a marked floor diagram $(\D,\mfk)$ is $s$-compatible if for any $\alpha \in S$, the set $\mfk^{-1}(\alpha)$ consists in either an edge and a floor, or two edges both entering or leaving the same floor.

We only deal with the case of Hirzebruch surfaces. The case of $h$-transverse, horizontal and non-singular toric surfaces is obtained as in the main body of this paper, by encoding the divergence of the floors of the diagrams via sloping pairs. All details can be found in \cite{mevel-2024-proprietes}.

\subsection{The genus $0$ case}

To take into account the parameter $s$ we define $s$-compatible words.

\begin{defi}
We say that a word $\sfW=\sfw_1\sfw_2\cdots$ is \emph{$s$-compatible} if for any $1\leq j\leq s$ we have $\sfw_{2j-1}=\sfw_{2j}$.
\end{defi}

The bijective correspondence between marked floor diagrams and words is as follows.

\begin{prop} \label{prop-W(Delta,s)}
Let $\D$ be a $s$-compatible marked floor diagram of Newton polygon    
$\Delta^\delta_{aE+bF}$, with $\codeg(\D) \leq i$ and 
$b \geq i+2s$. Then the word $W(\D)$ satisfies the following.
    \begin{itemize}
        \item[(i)-(iii)] from Proposition \ref{prop-corresp-word-diag-genus-0-hirzebruch} are still satisfied.
        \item[(iv)] The word is $s$-compatible.
    \end{itemize}
We denote by $\WWW(\Delta^\delta_{aE+bF},s)$ the set of words satisfying the above conditions. Given a word $\sfW\in\WWW(\Delta^\delta_{aE+bF},s)$, there is a unique way to recover a $s$-compatible marked floor diagram of Newton polygon $\Delta^\delta_{aE+bF}$.
\end{prop}

\begin{proof}
For the converse construction and proofs of items $(i)-(iii)$, see Proposition \ref{prop-corresp-word-diag-genus-0-hirzebruch}. 
For item $(iv)$, the first $2s$ marked points lie on ends because $b \geq i+2s$ and $\codeg(\D) \leq i$. As the diagram is $s$-compatible, for any $j\leq s$ the marked points $2j-1$ and $2j$ lie on ends adjacent to the same floor. Thus, the word is also $s$-compatible. 
\end{proof}

Recall we define the codegree of a word such that $\codeg(\D) = \codeg(W(\D))$.

\begin{lem} \label{lem-description-words-codeg-i}
Assume $i\geq 1$, $a>2i$, and $b \geq i+2s$.
The words in $\WWW(\Delta^\delta_{aE+bF},s)$ of codegree at most $i$ are of the form described by Lemma \ref{lem-BT-words-if-bounded-codegree}. Moreover, the word $\sfB_0$ is $s$-compatible.
\end{lem}

\begin{proof}
Let $\sfW$ be such a word. We only need to show that $\sfB_0$ is $s$-compatible. The hypothesis $b \geq i+2s$ together with $\codeg(\sfW)\leq i$ implies that the diagram corresponding to $\sfW$ has at least $2s$ infinite edges attached to the bottom floor. Hence the first $2s$ letters $\sfb_*$ are before the first letter $\sff$ in the word. Since a word in $\WWW(\Delta^\delta_{aE+bF},s)$ is $s$-compatible, then so is $\sfB_0$.
\end{proof}

As in the main body of this paper, the word is hence described by a core $(\sff\sfe)^{a-1}\sff$, a  $B$-word and a $T$-word. We still call $B$-words and $T$-words ``end-words'', and denote by $\S$ the set of sentences, \ie 
\[ \S=\{ (\sfS_0,\sfS_1^{(1)},\sfS_1^{(2)},\dots,\sfS_i^{(1)},\sfS_i^{(2)})\ |\ i\geq 0,\ \sfS_j^{(k)} \text{ word in }\{\sfs_*\}_{*\geq j} \}, \]
endowed with functions $\codeg,\ \ell_0,\ \ell_j^{(k)},\ \ell : \ \S \to \NN $.
For $n\geq0$ we denote by $\S(n)$ the set of sentences with total length $n$, and $\S^s(n)$ the subset of sentences with length $n$ such that $\sfS_0$ a is $s$-compatible word.

Lemma \ref{lem-description-words-codeg-i}
asserts that choosing a word $\sfW$ in $\WWW(\Delta^\delta_{aE+bF},s)$ having codegree at most $i$ and with $\delta,a,b$ large enough amounts to choose :
\begin{itemize}
    \item an element $\bfk \in \S^s(b)$ that encodes the $B$-words,
    \item an element $\tfk \in \S(b+\delta a)$ that encodes the $T$-words,
\end{itemize}
such that $\codeg(\sfW) = \codeg(\tfk) + \codeg(\bfk) \leq i$.
The computation of a generating series over $\WWW(\Delta^\delta_{aE+bF},s)$ hence splits into the computations of some generating series over $\S^s(b)$ and $\S(b+\delta a)$.
The determination of the generating series over the $s$-compatible words is an adaptation of Lemma \ref{lem-generating-series-T-words}, see Lemma \ref{app-lem-generating-series-T-words} below.

\begin{defi} \label{def-mult-sentence}
We define the \emph{$s$-multiplicity} of a sentence $\sfk \in \S^s(n)$ to be 
\[ \mu_{\S^s(n)}(\sfk)=(1-x)^n x^{\codeg(\sfk)} \left( \frac{1+x}{1-x}\right)^s . \]

\end{defi}

\begin{lem}\label{app-lem-generating-series-T-words}
Let $n>i\geqslant 1$. The generating series of $s$-compatible sentences of length $n$ counted with the $s$-multiplicity is
\[ \sum_{\sfk\in\S^s(n)}  \mu_{\S^s(n)}(\sfk) = p(x)^2 \mod x^{i+1}.\]
\end{lem}

\begin{proof}
The proof is similar to the one of lemma \ref{lem-generating-series-T-words}. We indicate what are the minor changes.

First, the sum of the multiplicities of sentences $\sfk$ such that $\ell_0(\sfk) = l_0$ and $\ell_j^{(k)}(\sfk) = l_j^{(k)}$ is
\[ (1-x)^n  \left(\frac{1+x}{1-x}\right)^s \left( \dsum_{\substack{\ell(\sfS_0)=l_0 \\ \sfS_0\ s\text{-compatible}}} x^{\codeg(\sfS_0)} \right) \times \dprod_{j,k} \left( \dsum_{\ell(\sfS_j^{(k)}) = l_j^{(k)}} x^{\codeg(\sfS_j^{(k)})} \right). \]

Second, the computation for letters in $\sfS_j^{(k)}$ is the same, but the one for letters in $\sfS_0$ changes a bit. Indeed, letters in $\sfS_0$ can take values in $\{\sfs_*\}_{*\geq 0}$, but they are not chosen independently since for any of the first $s$ pairs of letters, the letters of the pair have to take the same value. Thus, we get
\begin{align*}
    \dsum_{\substack{\ell(\sfS_0)=l_0 \\ \sfS_0\ s\text{-compatible}}} x^{\codeg(\sfS_0)} = & \left(\dsum_{k\geq0} x^{2k} \right)^{s}\left(\dsum_{k\geq0} x^k \right)^{l_0-2s} \\
    = & \left(\dfrac{1}{1-x^2}\right)^{s}\left(\dfrac{1}{1-x}\right)^{l_0-2s} \\
    = & \left(\frac{1-x}{1+x}\right)^s\left(\frac{1}{1-x}\right)^{l_0}.
\end{align*}

Hence, the term $\left(\frac{1+x}{1-x}\right)^s$ in the sum of the multiplicities cancels with the $\left(\frac{1-x}{1+x}\right)^s$ appearing in the words $\sfS_0$. The rest of the proof is as in Lemma \ref{lem-generating-series-T-words}.
\end{proof}

\begin{rem}
    Note that this does not depend on $s$. As a consequence, the asymptotic refined invariant in Theorem \ref{app-theo-AR-genus0-general} below is independent of $s$.
\end{rem}

In the main body of the present paper, the results are stated in terms of asymptotic refined invariants $AR_g^X$. In this appendix we similarly consider the asymptotic refined invariant $AR_{g,s}^X$,  
which amounts to count the floor diagrams with the multiplicity
\[ x^{\codeg(\D)} (1-x)^{2b+\delta a} \left(\dfrac{1+x}{1-x}\right)^s \dprod_{e\in E^0(\D)} (1-x^{w(e)})^2 ,\]
which turns to be equal to $\mu_{\S^s(b)}(\bfk) \mu_{\S^0(b+\delta a)}(\tfk) \mod x^i$, because one can assume that the weights of the bounded edges are large, see Lemma \ref{lem-multiplicity-asymptotic-word-genus-0-hirzebruch}.

The following theorem is proven in Theorem \ref{theo-AR-genus0-hirzebruch} for $s=0$.

\begin{theo}\label{app-theo-AR-genus0-Hirz}
    The genus $0$ asymptotic refined invariant of the Hirzebruch surface $\FF_\delta$ is 
    \[ AR^{\FF_\delta}_{0,s}= p(x)^4. \]
\end{theo}

\begin{proof}
The proof is as in Theorem \ref{theo-AR-genus0-hirzebruch}, except that 
\begin{itemize}
	\item the multiplicity is $(1-x)^{2b+\delta a} x^{\codeg(\sfW)} \left( \dfrac{1+x}{1-x} \right)^s  \mod x^{i+1}$,
	
	\item one has to replace $\S(b)$ by $\S^s(b)$,
	
	\item when factorizing the generating series, the first term is 
	\[ (1-x)^b \left(\frac{1+x}{1-x}\right)^s \sum_{\bfk\in\S^s(b) } x^{\codeg(\bfk)}, \]
	
	\item one has to use Lemma \ref{app-lem-generating-series-T-words} instead of Lemma \ref{lem-generating-series-T-words} to conclude.
\end{itemize}
\end{proof}

For $h$-transverse, non-singular and horizontal toric surface, Theorem \ref{theo-AR-genus0-general} adapts similarly to the following.

\begin{theo} \label{app-theo-AR-genus0-general}
Let $X$ be toric surface associated to a $h$-transverse, horizontal and non-singular polygon. The genus $0$ asymptotic refined invariant is given by
\[ AR^X_{0,s} = p(x)^\chi . \]
\end{theo}

\subsection{The genus $1$ case}

In the main body of this paper we introduced \emph{nerved diagrams} to build marked floor diagrams of genus $1$ from one of genus $0$. 
We also introduced measures $\mu_{\S(n)}$ and $\mu_\L$, such that if $\pi$ is the map $\pi = (\ell_j^{(k)})_{j,k} : \S(n) \to \L$ then $\mu_{\S(n)}( \pi^{-1}(\bfl)) = \mu_\L(\bfl)$. Here, we replace $\mu_{\S(n)}$ by $\mu_{\S^s(n)}$, see Definition \ref{def-mult-sentence}. The idea was then to introduce the normalized measures $\nu_{\S(n)}$ and $\nu_{\S^s(n)}$ and to see sums with multiplicities as integrals along these measures. We thus need to explain how the integrals computations change when $s$ is non-zero. The first difference appears when looking at the leak function.

\begin{defi} \label{def-leak-function}
We define on $\S^s(n)$ the \emph{leak function} $\phi_m^s[n](\sfk)$ equal to the number of letters of $\sfk$ with an index larger than $m$. To get a function of $\bfl\in\L$, we average over the set $\pi^{-1}(\bfl) \cap \S^s(n)$ of $s$-compatible sentences with lengths $\bfl$ :
\[ \varphi_m^s[n](\bfl) := \dfrac{1}{1-x^m}\frac{1}{\mu_{\S^s(n)}(\pi^{-1}(\bfl))}\int_{\pi^{-1}(\bfl) \cap \S^s(n)}\phi_m^s[n]\dd\mu_{\S^s(n)}. \]
\end{defi}

\begin{lem}\label{app-lem-expr-leak-function-in-l-monomials}
     We have the following expression on $\L$ :
     \[ \varphi^s_m[n](\bfl) = n \ang{m}-2sx^m+\psi_m(\bfl), \ \text{ where } \ \psi_m = \ang{m}\sum_{j=1}^m \frac{\ell_j}{\ang{j}}+\sum_{j=m+1}^{+\infty} \ell_j. \]
\end{lem}

\begin{proof}
We proceed as in the proof of Lemma \ref{lem-expr-leak-function-in-l-monomials}.
In terms of the letters, the leak function $\phi_m^s[n]$ is
        \[ \phi^s_m[n](\sfk)= \sum_{\sfs\in\sfS_0}\mathds{1}(p\geq m \text{ with }\sfs=\sfs_p) + \sum_{j,k}\sum_{\sfs\in\sfS_j^{(k)}}\mathds{1}(p\geq m \text{ with }\sfs=\sfs_p) \]
and we need to compute 
    \[ I_\sfs = (1-x)^n \left( \dfrac{1+x}{1-x} \right)^s \sum_{\sfk \in \pi^{-1}(\bfl)\cap\S^s(n)} \mathds{1}(p\geq m \text{ with }\sfs=\sfs_p) x^{\codeg(\sfk)}  \]
    for each term $\mathds{1}(p\geq m \text{ with }\sfs=\sfs_p)$  corresponding to a position of the letter $\sfs$ in one of the words $\sfS_0$ or $\sfS_j^{(k)}$.
    
    If $\bfl = (l_0,l_j^{(k)})$ then the sum splits into the product of sums
    \[ \dsum_{\substack{\ell(\sfS_0)=l_0 \\ \sfS_0\ s\text{-compatible}}}  \times \dprod_{j,k}  \dsum_{\ell(\sfS_j^{(k)}) = l_j^{(k)}} \]
    but the values of the letters are constrained by the condition $(p\geq m \text{ with } \sfs=\sfs_p)$.

    Assume first that the position corresponding to $\sfs$ is in $\sfS_j^{(k)}$. Then the computation is as in Lemma \ref{lem-expr-leak-function-in-l-monomials} except that the sum over $\sfS_0$ leads a factor $\left(\frac{1-x}{1+x}\right)^s$ as in Lemma \ref{app-lem-generating-series-T-words}, which cancels with the one in $I_\sfs$. In the end, 
    \[ I_\sfs = x^{(m-j)_+}\mu_\L(\bfl) .\]

    Assume now that the position of $\sfs$ is in $\sfS_0$. 
    If $\sfs$ is not in the first $2s$ letters then the computation is as above and 
    \[I_\sfs = x^m \mu_\L(\bfl) . \]
    
    If the position corresponding to $\sfs$ is among the first $2s$ letters, for $\sfS_0$ we get
    \begin{align*}
         & \left(\dsum_{k\geq0} x^{2k} \right)^{s-1} \left( \dsum_{p\geq0} \mathds{1}(p\geq m) x^{2p} \right)  \left(\dsum_{k\geq0} x^k \right)^{l_0-2s} \\
        = & \left(\dfrac{1}{1-x^2}\right)^{s-1} \left( \dfrac{x^{2m}}{1-x^2} \right)  \left(\dfrac{1}{1-x}\right)^{l_0-2s} \\
        = & \left(\frac{1-x}{1+x}\right)^s\left(\frac{1}{1-x}\right)^{l_0} x^{2m}.  
    \end{align*}
    This gives
    \[ I_\sfs  = x^{2m} \mu_\L(\bfl). \]
    
    Finally, the integral $I_\sfs$ is
     \[ I_\sfs =  \mu_\L(\bfl) \left\{ \begin{array}{ll}
            x^{2m} & \text{ if }\sfs\text{ in the first }2s\text{ letters,} \\
            x^{(m-j)_+} & \text{ else.} \\
        \end{array} \right. \]            
    Adding the above equality over all the letter positions in the word, and using the equality $\mu_{\S^s(n)}(\pi^{-1}(\bfl)) = \mu_\L(\bfl)$, we perform the computation as is Lemma \ref{lem-expr-leak-function-in-l-monomials}, but $l_0 x^m$ is replaced by 
    \[ 2s x^{2m} + (l_0-2s)x^m = 2s x^m(x^m-1) + l_0 x^m.\] 
    The first part give the term $-2s x^m$, while the second part is managed as in Lemma \ref{lem-expr-leak-function-in-l-monomials}.

\end{proof}

We can now compute the asymptotic refined invariant for the Hirzeburch surfaces. Theorem \ref{theo-AR-genus1-hirzebruch} becomes the following.

\begin{theo}\label{app-theo-AR-genus1-hirzebruch}
    The genus $1$ asymptotic refined invariant of the Hirzebruch surface $\FF_\delta$ is
        \[ AR_{1,s}^{\FF_\delta} = p(x)^4\left(g_{\max}+2s \frac{x}{1-x}-12 E_2(x) \right), \]
    where $\gmax(\Delta^\delta_{aE+bF}) = \frac{1}{2}(a-1)(2b+\delta a-2)$.
\end{theo}

\begin{proof}
The computation of the asymptotic refined invariant goes through three steps. 

\medskip

$\circ$ \textbf{First step: expression over $\S^s(b)\times\S(b+\delta a)$.}
It is the same as in Theorem \ref{theo-AR-genus1-hirzebruch}, replacing $\phi_m$ by $\phi_m^s$.

\medskip

$\circ$ \textbf{Second step: integration over $\S^s(b)\times\S(b+\delta a)$.}
For $i<m<a-i$ or $m\geq a-i$, the computations are as in Theorem \ref{theo-AR-genus1-hirzebruch}. If $m\leq i$ the computations are identical, up to the correction term $2sx^m$.

\medskip

$\circ$ \textbf{Third step: summation over the values of $m$.}
The sum for $m\geq 1$ of the correction term $2sx^m$ gives $2s \frac{x}{1-x}$, while the rest of the calculation is the same as in Theorem \ref{theo-AR-genus1-hirzebruch}.
\end{proof}

For $h$-transverse, non-singular and horizontal toric surface, we can copy the proof of Theorem \ref{theo-AR-genus1-general} and add the term $2s \frac{x}{1-x}$ where necessary It gives the following.

\begin{theo}
Let $X$ be toric surface associated to a $h$-transverse, horizontal and non-singular polygon. Let $\gmax$ be the function $\Delta \mapsto \gmax(\Delta)$. The genus $1$ asymptotic refined invariant is given by
\[ AR^X_{1,s} = p(x)^\chi \left(g_{\max} +2s \frac{x}{1-x}-12E_2(x) \right). \]
\end{theo}

\bibliographystyle{alpha}
\bibliography{biblio}

\end{document}